\documentclass[12pt]{amsart}
\usepackage{amsmath}
\usepackage{amssymb}

\usepackage{overpic,graphicx,mathrsfs}

\theoremstyle{plain}
\newtheorem{lemma}{Lemma}[section]
\newtheorem{proposition}[lemma]{Proposition}
\newtheorem{theorem}[lemma]{Theorem}

\newtheorem{conjecture}[lemma]{Conjecture}

\theoremstyle{definition}
\newtheorem{definition}[lemma]{Definition}

\def\C{\mathbb C}
\def\R{\mathbb{R}}
\def\Z{\mathbb{Z}}
\def\to{\rightarrow}
\def\SL{\mathrm{SL}}
\def\tr{\mathrm{tr}}
\def\S{\mathbb{S}}
\def\K{\mathcal{K}}
\def\A{\mathcal{A}}
\def\I{\mathcal{I}}
\def\l{\langle\!\langle}
\def\r{\rangle\!\rangle}
\def\t{x}

\title{Trace-free characters and abelian knot contact homology I}

\author{Fumikazu Nagasato}

\address{Department of Mathematics, Meijo University, Tempaku, Nagoya 468-8502, Japan}

\email{fukky@meijo-u.ac.jp}

\subjclass[2020]{Primary 57K31; Secondary 57K18}

\keywords{abelian knot contact homology, character varieties, trace-free characters}

\begin{document}

\maketitle

\begin{abstract}
We study the structure underlying Ng's conjecture, 
which relates the degree $0$ abelian knot contact homology of a knot $K$
to the coordinate ring of the $\SL_2(\C)$-character variety $X(\Sigma_2 K)$
of the $2$-fold branched cover of the $3$-sphere branched along $K$. 
Our approach is based on the study of (meridionally) trace-free characters of knot groups. 
For each knot $K$, they form a closed algebraic subset $S_0(K)$ 
of the $\SL_2(\C)$-character variety of $K$, defined by the trace-free condition on meridians. 
The subset $S_0(K)$, called the trace-free slice of $K$, has a natural connection 
to $X(\Sigma_2K)$. 
We show that the trace-free slice admits the structure of a $2$-fold branched cover 
of a closed algebraic set, called the fundamental variety, whose coordinate ring 
coincides with the nilradical quotient of the complexification 
of degree $0$ abelian knot contact homology. 
Using this framework, we introduce the notion of \emph{ghost characters} 
and prove that Ng's conjecture holds for a knot $K$ 
if and only if $K$ admits no ghost characters. 
This criterion establishes Ng's conjecture for all 2-bridge and 3-bridge knots.
\end{abstract}


\section{Introduction}\label{sec_intro}
The purpose of this paper is to analyze the structure underlying Ng's conjecture, 
which relates the character variety of the 2-fold branched cover $\Sigma_2K$ 
of the 3-sphere $\S^3$ branched along a knot $K$ 
and the degree 0 abelian knot contact homology of $K$. 
The key object in our approach is the characters of (meridionally) trace-free 
$\SL_2(\C)$-representations of the knot group $G(K)$. 
These characters form a closed algebraic subset $S_0(K)$ of the $\SL_2(\C)$-character variety 
$X(K)$ of $G(K)$, defined by the trace-free condition on meridians. 
We call the subset $S_0(K)$ the trace-free slice of $X(K)$ (or simply the trace-free slice of $K$).
The defining equations for $S_0(K)$ can be obtained directly from a Wirtinger presentation 
of $G(K)$ (see Theorem \ref{defpoly_S0K}; cf.\ \cite[Theorem 1.1]{Nagasato4}).

Character varieties play an important role in both geometry and topology, 
and the trace-free slice $S_0(K)$ is no exception. 
A classical example is the Casson-Lin invariant of a knot $K$, 
introduced in \cite{Lin1} via trace-free $\mathrm{SU}(2)$-representations of the knot group $G(K)$.
This invariant provides a gauge-theoretic interpretation of the knot signature. 
Our interest in $S_0(K)$ originates from the study of an $\SL_2(\C)$-analogue 
of the Casson-Lin invariant (see \cite{Nagasato1, Nagasato0, Nagasato-Yamaguchi}, etc). 
In these earlier works, we established a precise relationship between 
trace-free $\SL_2(\C)$-representations of $G(K)$ and $\SL_2(\C)$-representations 
of the fundamental group $\pi_1(\Sigma_2 K)$. 
Based on this relationship, we introduced a natural map 
\[
\widehat{\Phi}: S_0(K) \longrightarrow X(\Sigma_2 K),
\]
where  $X(\Sigma_2 K)$ denotes the $\SL_2(\C)$-character variety of $\pi_1(\Sigma_2K)$ 
(see \cite{Nagasato-Yamaguchi} for details). 

Degree $0$ knot contact homology $HC_0(K)$ was introduced by L. Ng \cite{Ng1} 
combinatorially using the knot DGA. Ekholm, Etnyre, Ng, and Sullivan \cite{EENS} 
developed a fully geometric and topological framework realizing this combinatorial theory. 
The homology $HC_0(K)$ has several striking properties. 
One of the most interesting properties is the existence of an isomorphism, 
shown in \cite[Proposition 5.6]{Ng2}, between the complexification of the abelianization 
$HC_0^{ab}(K)$ and the coordinate ring $\mathbf{C}[X(\Sigma_2K)]$ 
of the character variety $X(\Sigma_2K)$ for any $2$-bridge knot $K$ 
(see also \cite{Nagasato, Nagasato3, Nagasato4} for alternative proofs). 
Based on this observation, in \cite[Conjecture 5.7]{Ng2}, Ng conjectured that 
such an isomorphism exists for any knot. 
Despite the existence of such an isomorphism for $2$-bridge knots, 
the underlying structure of this correspondence is not yet well understood, 
and Ng's conjecture remains open in general. 

In this context, the trace-free slice $S_0(K)$ provides a natural algebro-geometric framework 
for describing the correspondence between $HC_0^{ab}(K)\otimes \C$ 
and $\mathbf{C}[X(\Sigma_2K)]$. 
To clarify the mechanism underlying this relationship, we analyze the structure of $S_0(K)$ 
as a $2$-fold branched cover over a closed algebraic set $F_2(K)$, 
referred to as the fundamental variety of $K$. 
In fact, the coordinate ring $\mathbf{C}[F_2(K)]$ coincides 
with $(HC_0^{ab}(K)\otimes \C)/\sqrt{0}$. 
By Hilbert's Nullstellensatz, Ng's conjecture is therefore equivalent to the statement 
that $F_2(K)$ and $X(\Sigma_2K)$ are isomorphic as algebraic sets (Conjecture \ref{conj_nag}). 
Using the above formulation of $S_0(K)$, we introduce the notion of a \emph{ghost character} 
of a knot (Definition \ref{def_ghost}), namely a point in $F_2(K)$ that does not lift to $S_0(K)$. 
Then, the surjectivity of the map $\widehat{\Phi}: S_0(K) \to X(\Sigma_2K)$ implies that 
Ng's conjecture holds for knots admitting no ghost characters 
(Theorem \ref{thm_Ng_conj} (1)). 
In particular, this is the case for all 2-bridge and 3-bridge knots (Theorem \ref{thm_noghost}). 
Moreover, it follows that Ng's conjecture holds for a knot $K$ if and only if $K$ 
admits no ghost characters (Theorem \ref{iff}). 

The paper is organized as follows. 
In Section \ref{sec_review}, we briefly review $\SL_2(\C)$-character varieties of 
finitely presented groups and introduce the trace-free slice $S_0(K)$. 
In Section \ref{sec_main}, we derive defining equations for $S_0(K)$ 
via the Kauffman bracket skein algebra and describe the structure of $S_0(K)$ 
as a $2$-fold branched cover $q : S_0(K) \to F_2(K)$. 
In Section \ref{sec_hc0}, we first review degree $0$ abelian knot contact homology 
$HC_0^{ab}(K)$ and show that $(HC_0^{ab}(K) \otimes \C)/\sqrt{0}$ 
and $\mathbf{C}[F_2(K)]$ coincide. 
Using Hilbert's Nullstellensatz, we then reformulate Ng's conjecture 
in terms of the projection $h^*: X(\Sigma_2K) \to F_2(K)$ 
by analyzing the structure of $X(\Sigma_2K)$ and the map 
$\widehat{\Phi}: S_0(K) \to X(\Sigma_2K)$. 
Using this formulation, we introduce ghost characters and show that all $2$-bridge and 
$3$-bridge knots admit no ghost characters, thereby verifying Ng's conjecture for these knots.  
Finally, we prove that Ng's conjecture holds for a knot $K$ if and only if $K$ 
admits no ghost characters.


\section{Trace-free slice of a knot}\label{sec_review}

\subsection{Character variety of a finitely presented group}
We recall the definition of $\SL_2(\C)$-character varieties (or simply character varieties) 
for finitely presented groups, following the work of Culler and Shalen \cite{Culler-Shalen}. 
Let $G$ be a finitely presented group with generators $g_1,\cdots,g_n$. 
For a representation $\rho:G \to \SL_2(\C)$, 
the character $\chi_{\rho}$ is defined by $\chi_{\rho}(g)=\tr(\rho(g))$ for $g\in G$. 
The $\SL_2(\C)$-trace identity
\[
\tr(AB)=\tr(A)\tr(B)-\tr(AB^{-1})
\]
for $A,B \in \SL_2(\C)$ 
implies that, for any $g \in G$, the trace $\tr(\rho(g))$ can be expressed 
as a polynomial in the following trace functions
\[ 
\{t_{g_i}(\rho)\}_{1\leq i\leq n},\ 
\{t_{g_ig_j}(\rho)\}_{1\leq i<j \leq n},\ 
\{t_{g_ig_jg_k}(\rho)\}_{1 \leq i<j<k \leq n},
\] 
where $t_{g}(\rho):=\tr(\rho(g))$ 
(see \cite{Culler-Shalen,Fricke, Horowitz,Vogt}; cf. \cite{Gonzalez-Montesinos}). 

Let $\mathfrak{X}(G)$ denote the set of the characters of $\SL_2(\C)$-representations of $G$.  
In this setting, the $\SL_2(\C)$-character variety (or simply the character variety) $X(G)$ 
of $G$ is defined as the image under the map $t$
\[
t:\mathfrak{X}(G) \to \C^{n+{n\choose 2}+{n\choose 3}},\ 
t(\chi_{\rho})=\left(t_{g_i}(\chi_{\rho}); t_{g_ig_j}(\chi_{\rho}); t_{g_ig_jg_k}(\chi_{\rho})\right), 
\]
where $t_{g}(\chi_{\rho}):=t_{g}(\rho)$. 
It follows that $X(G)$ is a closed algebraic subset of the affine space.
The parametrization of $X(G)$ depend on the choice of a generating set, 
but only up to biregular equivalence.
Therefore, $X(G)$ is an invariant of $G$ up to biregular equivalence.

We now turn to the character varieties of knot groups. 
For a knot $K$ in $\S^3$, we denote by $E_K$ the knot exterior, 
and by $G(K)=\pi_1(E_K)$ the knot group. The group $G(K)$ has a presentation 
generated by meridians of $K$. For instance, given a knot diagram $D_K$ with $n$ crossings, 
the Wirtinger algorithm yields the Wirtinger presentation associated with $D_K$: 
\[
G(K)=\langle m_1,\cdots, m_n \mid r_1,\cdots,r_n \rangle, 
\]
where $m_i$ $(1 \leq i \leq n)$ is a meridian corresponding to the $i$th arc of $D_K$, 
and $r_j$ $(1 \leq j \leq n)$ is a word in $m_1,\cdots,m_n$ associated with the $j$th crossing 
(see, for example, \cite{Burde-Zieschang, Kawauchi} etc.). 
If the $s$th crossing in $D_K$ is depicted as 
\[
\begin{minipage}{3.5cm}
\begin{overpic}[width=\hsize]{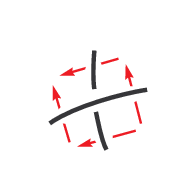}
\put(30,70){$m_j$}
\put(10,40){$m_i$}
\put(80,30){$m_i$}
\put(60,10){$m_k$}
\end{overpic}
\end{minipage} 
\]
then we call the triple $(i,j,k)$ (with $j<k$) a Wirtinger triple of $D_K$. 
In this case, one may take $r_s=m_im_jm_i^{-1}m_k^{-1}$. 
Any one of the relations $r_1,\cdots,r_n$ can be eliminated due to deficiency one property 
of knot groups. 
Furthemore, if $K$ is a $b$-bridge knot, then we can obtain a $b$-bridge knot group 
presentation $G(K)=\langle m_1,\cdots,m_b \mid \tilde{r}_1,\cdots, \tilde{r}_{b-1} \rangle$ 
by reducing the generators and the relators of the Wirtinger presentation by Tietze transformations. 
The character variety $X(K)=X(G(K))$ associated 
with a presentation generated by $n$ meridians $m_1,\cdots,m_n$ is given by  
\[
X(K)=t(\mathfrak{X}(G(K)))
=\{(t_{m_i}(\chi_{\rho}); t_{m_im_j}(\chi_{\rho}); t_{m_im_jm_k}(\chi_{\rho})) 
\in \C^{n+{n \choose 2}+{n \choose 3}} \mid \chi_{\rho} \in \mathfrak{X}(G(K))\}. 
\]
Since all meridians are conjugate,  
the entries $(t_{m_i}(\chi_{\rho}))$ may be replaced by a single representative  
$t_{m_j}(\chi_{\rho})$ for some $j$. 
The character variety $X(K)$ thus provides an invariant of knot groups,
and therefore an invariant of knots themselves, up to biregular equivalence.


\subsection{Trace-free slice of a knot}
We focus on a special class of representations of knot groups. 
Let $\mu_K$ be a meridian of $K$. A representation $\rho: G(K) \to \SL_2(\C)$ 
is said to be (meridionally) trace-free  if $\tr(\rho(\mu_K))=0$ 
holds\footnote{This is also 
called traceless representation.}. 
The character of a trace-free representation is called a trace-free character. 
The set of trace-free characters forms a subset 
\[
\mathfrak{S}_0(K):=\{\chi_{\rho}\in \mathfrak{X}(K) \mid \chi_{\rho}(\mu_K)=0\}. 
\]
of $\mathfrak{X}(K)=\mathfrak{X}(G(K))$. 
By \cite{Culler-Shalen}, this subset can be realized 
as a closed algebraic subset of the character variety $X(K)$. 
Indeed, by definition, the image $t(\mathfrak{S}_0(K))$ corresponds to the intersection 
of $X(K)$ with the hyperplane defined by $t_{\mu_K}(\chi_{\rho})=0$. 
Since any meridians are conjugate, for a presentation generated by $n$ meridians 
$m_1,\cdots,m_n$, the equation $t_{\mu_K}(\chi_{\rho})=0$ implies   
$t_{m_i}(\chi_{\rho})=0$ for all $1 \leq i \leq n$. 
Projecting this hyperplane section in $\C^{1+{n \choose 2}+{n \choose 3}}$ 
to $\C^{{n \choose 2}+{n \choose 3}}$ via the map $p$ defined by 
\[
(z_1,\cdots,z_n; z_{12},\cdots,z_{nn-1}; z_{123},\cdots,z_{n-2,n-1,n}) 
\mapsto
(z_{12},\cdots,z_{nn-1}; z_{123},\cdots, z_{n-2,n-1,n}),  
\]
we obtain 
\[
p(t(\mathfrak{S}_0(K))) =\left\{\left.\left(t_{m_im_j}(\chi_{\rho})); t_{m_im_jm_k}(\chi_{\rho})\right)
\in \C^{{n\choose 2}+{n\choose 3}} \right| \chi_{\rho}\in\mathfrak{S}_0(K) \right\}. 
\]
Bacause $t_{m_i}(\chi_{\rho})=0$ for all $1 \leq i \leq n$, the projection $p$ is biregular, 
and the image is a closed algebraic set isomorphic to $t(\mathfrak{S}_0(K))$. 
We call this closed algebraic set the trace-free slice of $X(K)$ 
(or simply the trace-free slice of $K$) and denote it by $S_0(K)$. 
By construction, the trace-free slice is a knot invariant, up to biregular equivalence. 

The trace-free slices encode several kinds of topological information on knots.
In particular, they exhibit additivity under the connected sum \cite{Nagasato1},
are closely related to the knot signature \cite{Lin1} and the $A$-polynomial \cite{Nagasato2},
describe the structure of the 2-fold branched cover whose branch set consists of
the characters of metabelian representations \cite{Lin2, Nagasato-Yamaguchi},
and correspond to degree $0$ abelian knot contact homology \cite{Nagasato3, Nagasato0}.
The results of \cite{Nagasato3,Nagasato4,Nagasato0} were later reformulated 
in terms of reflective augmentations in \cite{Cornwell}. 

In this paper, we analyze the trace-free slice $S_0(K)$ in order to clarify 
the mechanism underlying Ng's conjecture. 
A key step in this approach is to realize $S_0(K)$ concretely as a closed algebraic set.
\begin{theorem}[cf. \cite{Nagasato4}, Theorem 3.2 in \cite{Gonzalez-Montesinos}]
\label{defpoly_S0K}
Let $G(K)=\langle m_1,\cdots,m_n \mid r_1,\cdots,r_{n-1} \rangle$ be 
a Wirtinger presentation. Then the trace-free slice $S_0(K)$ is isomorphic to 
the algebraic set: 
\[
S_0(K) \cong 
\left\{\left.
(x_{12},\cdots,x_{nn-1}; x_{123},\cdots,x_{n-2,n-1,n})\in\C^{{n \choose 2}+{n \choose 3}}
\right| {\rm (F2)},{\rm (GH)} \right\},
\]
where the defining equations are given as follows: 
\begin{description}
\item[(F2)] Fundamental relations 
\begin{eqnarray*}
&x_{ak}=x_{ij}x_{ai}-x_{aj},&\\
&(1 \leq a \leq n,\ (i,j,k):\mbox{a Wirtinger triple}),&
\end{eqnarray*}
\item[(GH)] General hexagon relations 
\begin{eqnarray*}
&x_{i_1 i_2 i_3} \cdot x_{j_1 j_2 j_3}
=\frac{1}{2}
\left|
\begin{array}{ccc}
x_{i_1 j_1} & x_{i_1 j_2} & x_{i_1 j_3}\\
x_{i_2 j_1} & x_{i_2 j_2} & x_{i_2 j_3}\\
x_{i_3 j_1} & x_{i_3 j_2} & x_{i_3 j_3}
\end{array}\right|,&\\
&(1 \leq i_1<i_2<i_3 \leq n,\ 1 \leq j_1<j_2<j_3 \leq n),&
\end{eqnarray*}
\end{description}
with the convention $x_{ii}=2$, $x_{ji}=x_{ij}$ and 
$x_{i_{\sigma(1)}i_{\sigma(2)}i_{\sigma(3)}}=\mathrm{sign}(\sigma)x_{i_1i_2i_3}$ 
for any permutation $\sigma \in \mathfrak{S}_3$. 
\end{theorem}

The coordinates $x_{ij}$ and $x_{ijk}$ correspond to 
$-t_{m_im_j}(\chi_{\rho})$ and $-t_{m_im_jm_k}(\chi_{\rho})$, respectively. 
(See Theorem \ref{thm_bps} for an explanation of the negative signs in these trace functions.) 
The relation {\rm (F2)} is symmetric in the indices $j$ and $k$, 
which correspond to the underarcs. Indeed, for a Wirtinger triple $(i,j,k)$, 
the relation {\rm (F2)} with $a=i$ implies $x_{ik}=x_{ij}$.

We remark that, in \cite{Nagasato4}, an alternative system of defining equations for $S_0(K)$ 
was given. This system consists of (F2), the rectangle relations (R), and the hexagon relations (H). 
In the present paper, the relations (R) and (H) are replaced by all sister relations (GH) 
associated with (H) (see Subsection \ref{subsec_gam} for details).


\section{Proof of Theorem \ref{defpoly_S0K}}\label{sec_main}
The proof of Theorem \ref{defpoly_S0K} relies essentially on the work of
Gonz\'alez-Acu\~na and Montesinos \cite[p. 639]{Gonzalez-Montesinos}.
They introduced a family of polynomials whose common zero locus gives
the character variety $X(F)$ of a finitely generated free group $F$.
For completeness, we briefly recall their construction and then formulate
its trace-free analogue, which will play a key role in the proof of
Theorem \ref{defpoly_S0K}.

\subsection{Character varieties of free groups and their trace-free slices}\label{subsec_gam}
Let $F_n=\langle m_1,\cdots,m_n \rangle$ be a finitely generated free group.
Denote by $X(F_n)$ its character variety, defined as the image of
$\mathfrak{X}(F_n)$ under the trace map
$t: \mathfrak{X}(F_n) \to \C^{n+{n \choose 2}+{n \choose 3}}$.
In Subsection \ref{subsec_sketch}, we make use of the Kauffman bracket skein algebra
(abbreviated KBSA) to prove Theorem \ref{defpoly_S0K} (see Theorem \ref{thm_bps}).
Accordingly, instead of the original trace map $t$, we adopt the modified trace map
$\tilde{t}: \mathfrak{X}(F_n) \to \C^{n+{n \choose 2}+{n \choose 3}}$
defined by the negative traces: 
\[
\tilde{t}(\chi_{\rho})=(-t_{m_i}(\chi_{\rho}); -t_{m_im_j}(\chi_{\rho}); -t_{m_im_jm_k}(\chi_{\rho})).
\]
Following \cite{Gonzalez-Montesinos}, the character variety 
$X(F_n) \subset \C^{n+{n \choose 2}+{n \choose 3}}$ (with coordinates $(\t_i; \t_{ij}; \t_{ijk})$) 
is described as the common zero locus of the following polynomials (P1), $\cdots$, (P4). 

\begin{description}
\item[(P1)] 
The quadratic relations
\[
(\t_{abc})^2+P_{abc} \t_{abc}+Q_{abc},
\] 
for $1 \leq a < b < c \leq n$, where  
\begin{eqnarray*}
P_{abc} & = & \t_a \t_{bc}+\t_b \t_{ac} + \t_c \t_{ab} + \t_a \t_b \t_c,\\
Q_{abc} & = & (\t_a)^2+(\t_b)^2+(\t_c)^2+(\t_{ab})^2+(\t_{ac})^2+(\t_{bc})^2
-\t_{ab}\t_{ac}\t_{bc}\\
&&+\t_a \t_b \t_{ab}+\t_a \t_c \t_{ac} +\t_b \t_c \t_{bc}-4.
\end{eqnarray*}

\item[(P2)] 
The determinants $Q_{12ab}^{12ab}$ for $3 \leq a < b \leq n$, where 
\[
R_{abcd}^{ijkl}=\left|\begin{array}{cccc}
M_{ai} & M_{aj} & M_{ak} & M_{al}\\
M_{bi} & M_{bj} & M_{bk} & M_{bl}\\
M_{ci} & M_{cj} & M_{ck} & M_{cl}\\
M_{di} & M_{dj} & M_{dk} & M_{dl}
\end{array}\right|,
\ 
M_{ij}=-x_{ij}-\frac{1}{2}x_i x_j.  
\]

\item[(P3)] 
The determinants $R_{123a}^{123b}$ for any $4 \leq a < b \leq n$. \\

\item[(P4)] The following polynomials for $1 \leq a < b< c \leq n$, 
\[
(-\t_{123}+\t_{132})(-2\t_{abc}-\t_a \t_b \t_c-\t_a \t_{bc}-\t_b \t_{ac}-\t_c \t_{ab})\\
-   
\left|\begin{array}{cccc}
-\t_{1} & -\t_{1a} & -\t_{1b} & -\t_{1c}\\
-\t_{2} & -\t_{2a} & -\t_{2b} & -\t_{2c}\\
-\t_{3} & -\t_{3a} & -\t_{3b} & -\t_{3c}\\
2      & -\t_{a}  & -\t_{b} & -\t_{c}
\end{array}\right|.
\]
\end{description}

A remark on the indices is in order. 
The indices $1,2,3$ in (P1) through (P4) may be replaced 
by arbitrary $i,j,k$ with $1 \leq i < j < k \leq n$.
The resulting polynomials, called the sister relations of
(P1),$\cdots$,(P4), vanish identically on $X(F_n)$. 

Now, the trace-free slice of the free group $F_n$, denoted by $S_0(F_n)$, 
is defined as the intersection of $X(F_n)$ with the hyperplane $x_i=0$ 
for $1 \leq i \leq n$. Its defining polynomials are obtained by substituting 
$x_i=0$ into (P1) through (P4).

\begin{itemize}
\item (Triangle relations) We have $P_{abc}=0$ and 
\[
Q_{abc}=-\frac{1}{2}
\left|\begin{array}{ccc}
x_{aa} & x_{ab} & x_{ac}\\
x_{ba} & x_{bb} & x_{bc}\\
x_{ca} & x_{cb} & x_{cc}
\end{array}
\right|, 
\]
so that by (P1) we obtain   
\[
\mathbf{(T)}:\hspace*{0.5cm}
(\t_{abc})^2
-\frac{1}{2}
\left|\begin{array}{ccc}
x_{aa} & x_{ab} & x_{ac}\\
x_{ba} & x_{bb} & x_{bc}\\
x_{ca} & x_{cb} & x_{cc}
\end{array}
\right| \ (1 \leq a < b < c \leq n).
\]

\item (Rectangle relations) By (P2), we obtain 
\[
\mathbf{(R)}:\hspace*{0.5cm}
\left|\begin{array}{cccc}
x_{11} & x_{12} & x_{1a} & x_{1b}\\
x_{21} & x_{22} & x_{2a} & x_{2b}\\
x_{a1} & x_{a2} & x_{aa} & x_{ab}\\
x_{b1} & x_{b2} & x_{ba} & x_{bb}
\end{array}\right| \ (3 \leq a < b \leq n). 
\]

\item (Pentagon relations) By (P3), we obtain 
\[
\mathbf{(P)}
:\hspace*{0.5cm}
\left|\begin{array}{cccc}
x_{11} & x_{12} & x_{13} & x_{1a}\\
x_{21} & x_{22} & x_{23} & x_{2a}\\
x_{31} & x_{32} & x_{33} & x_{3a}\\
x_{b1} & x_{b2} & x_{b3} & x_{ba}
\end{array}\right| \ (4 \leq a < b \leq n).
\]

\item (Hexagon relations) 
By (P4), for any $1 \leq i < j < k \leq n$ and $1 \leq a < b < c \leq n$, we obtain  
\[
\mathbf{(H)}:\hspace*{0.5cm}
x_{123}x_{abc}-\frac{1}{2}
\left|\begin{array}{ccc}
x_{1a} & x_{1b} & x_{1c}\\
x_{2a} & x_{2b} & x_{2c}\\
x_{3a} & x_{3b} & x_{3c}
\end{array}
\right| \ 
(1 \leq a < b < c \leq n). 
\]

\end{itemize}
Thus the trace-free slice $S_0(F_n)$ is isomorphic to the common zeros locus 
of (T), (R), (P), and (H). 

We next reformulate these polynomials in a more systematic way. 
Set  
\[
D_{i_1i_2i_3}^{j_1j_2j_3}({\mathbf x})
:=
\left|\begin{array}{ccc}
x_{i_1j_1} & x_{i_1j_2} & x_{i_1j_3}\\
x_{i_2j_1} & x_{i_2j_2} & x_{i_2j_3}\\
x_{i_3j_1} & x_{i_3j_2} & x_{i_3j_3}
\end{array}\right|,\hspace*{0.5cm}
D_{i_1i_2i_3i_4}^{j_1j_2j_3j_4}({\mathbf x})
:=
\left|\begin{array}{cccc}
x_{i_1j_1} & x_{i_1j_2} & x_{i_1j_3} & x_{i_1j_4}\\
x_{i_2j_1} & x_{i_2j_2} & x_{i_2j_3} & x_{i_2j_4}\\
x_{i_3j_1} & x_{i_3j_2} & x_{i_3j_3} & x_{i_3j_4}\\
x_{i_4j_1} & x_{i_4j_2} & x_{i_4j_3} & x_{i_4j_4}
\end{array}\right| 
\]
for $\mathbf{x}=(x_{ij})$. 
By combining (T) and (H) with all their sister relations, 
we obtain a single family of polynomial relations:
\[
\mathbf{(GH)}:\hspace*{0.5cm}
x_{ijk} x_{abc}-\frac{1}{2}D_{ijk}^{abc}({\mathbf x})
\  
(1 \leq i < j < k \leq n,\ 1 \leq a < b < c \leq n),
\]
which we refer to as the general hexagon relations.  
It then follows that the pentagon relations (P) vanish 
whenever both ${\rm (GH)}=0$ and ${\rm (R)}=0$ hold.
Indeed, if $x_{123}=x_{12a}=x_{13a}=x_{23a}=0$, 
then by the cofactor expansion and ${\rm (H)}=0$, we have  
\[
{\rm (P)}=D_{123a}^{123b}({\mathbf x})
=\frac{x_{123}}{2}(-x_{b1}x_{23a}+x_{b2}x_{13a}-x_{b3}x_{12a}+x_{ba}x_{123})=0. 
\]
Hence, we may assume that at least one of $x_{123}$, $x_{12a}$, $x_{13a}$, $x_{23a}$ 
is nonzero. 
Suppose $x_{123} \neq 0$. Then, by ${\rm (R)}=0$ and ${\rm (GH)}=0$, we obtain 
\[
D_{123a}^{123a}({\mathbf x})
=\frac{x_{123}}{2}(-x_{a1}x_{23a}+x_{a2}x_{13a}-x_{a3}x_{12a}+2x_{123})=0,
\]
which shows that at least one of $x_{12a}$, $x_{13a}$, $x_{23a}$ is nonzero. 
If $x_{12a} \neq 0$, then ${\rm (GH)}=0$ and ${\rm (R)}=0$ yield 
\begin{eqnarray*}
x_{12a}
D_{123a}^{123b}({\mathbf x})
&=&\frac{x_{12a}x_{12b}}{2}(x_{31}x_{23a}-x_{32}x_{13a}+2x_{12a}-x_{3a}x_{123})\\
&=&x_{12b}
D_{123a}^{123a}({\mathbf x})
=0.
\end{eqnarray*}
Therefore, ${\rm (P)}=0$. The remaining cases can be shown similarly. 

Furthermore, the rectangle relations ${\rm (R)}=0$ also follow from ${\rm (GH)}=0$.
More precisely, any point $(x_{ij};x_{ijk}) \in \C^{{n \choose 2}+{n \choose 3}}$ 
satisfying ${\rm (GH)}=0$ also satisfies ${\rm (R)}=0$.
For example, suppose at least one of $x_{12i}$, $x_{12j}$, $x_{1ij}$, or $x_{2ij}$ 
vanishes (say $x_{12i}=0$).
Then the cofactor expansion of $D_{12ij}^{12ij}({\mathbf x})$ for ${\mathbf x}=(x_{ij})$ gives
\begin{eqnarray*}
D_{12ij}^{12ij}({\mathbf x}) &=& -x_{1j}D_{12i}^{2ij}({\mathbf x})+x_{2j}D_{12i}^{1ij}({\mathbf x})
-x_{ij}D_{12i}^{12j}({\mathbf x}) +2D_{12i}^{12i}({\mathbf x})\\
&=& 2x_{12i}(-x_{1j}x_{2ij}+x_{2j}x_{1ij}-x_{ij}x_{12j}+2x_{12i})=0. 
\end{eqnarray*}
If all of $x_{12i}$, $x_{12j}$, $x_{1ij}$ and $x_{2ij}$ are nonzero, then  
\begin{eqnarray*}
x_{12j} D_{12ij}^{12ij}({\mathbf x}) &=& 2x_{12j}x_{12i}(-x_{1j}x_{2ij}+x_{2j}x_{1ij}-x_{ij}x_{12j}+2x_{12i})\\
&=& x_{12i}(-x_{1j}D_{12j}^{2ij}({\mathbf x})+x_{2j}D_{12j}^{1ij}({\mathbf x})
-x_{ij}D_{12j}^{12j}({\mathbf x}) +2D_{12j}^{12i}({\mathbf x}))\\
&=&x_{12i}D_{12ij}^{12jj}({\mathbf x})=0. 
\end{eqnarray*}
Since $x_{12j} \neq 0$, it follows that ${\rm (R)}=D_{12ij}^{12ij}({\mathbf x})=0$. 
The remaining cases are analogous. We therefore obtain the following proposition. 

\begin{proposition}[cf. \cite{Gonzalez-Montesinos}]\label{thm_RH}
For a free group $F_n=\langle m_1,\cdots, m_n \rangle$, 
the trace-free slice $S_0(F_n) \subset \C^{{n \choose 2}+{n \choose 3}}$ 
is isomorphic to the common zeros locus of {\rm (GH)}.
\end{proposition}


\subsection{Proof of Theorem \ref{defpoly_S0K}}\label{subsec_sketch}
To prove Theorem \ref{defpoly_S0K}, we use the Kauffman bracket skein algebra 
(KBSA)\footnote{This is the specialization of the Kauffman bracket skein module 
at the parameter $t=-1$.} \cite{Bullock,Przytycki1,Przytycki2,Przytycki-Sikora}.  
The KBSA of a 3-manifold $M$, denoted by $\K_{-1}(M)$, is the quotient 
of the algebra over $\C$ generated by all free homotopy classes of loops in $M$ 
by the Kauffman bracket skein relations (specialized at $t=-1$):
\[
\begin{minipage}{7cm}\begin{overpic}[width=\hsize]{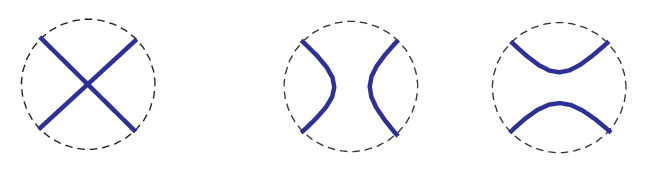}
\put(35,13){$-$}
\put(28,13){$=$}
\put(68,13){$-$}
\end{overpic}
\end{minipage},\ 
\begin{minipage}{2cm}\includegraphics[width=\hsize]{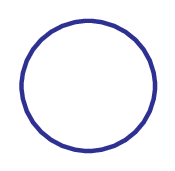}\end{minipage}
=-2,
\]
where in the first relation the loops outside the dashed circles coincide with each other. 
The product of two loops is defined by the disjoint union of them. 
In fact, a loop (a homotopy class of a loop) $s \in \K_{-1}(M)$ 
has the same properties as $-\tr(\rho(s))$ for an unspecified representation 
$\rho: \pi_1(M) \to \SL_2(\C)$. 
This gives rise to a correspondence between $\K_{-1}(M)$ 
and the coordinate ring of the character variety $X(\pi_1(M))$ 
(see Theorem \ref{thm_bps}). 
Here the coordinate ring $\mathbf{C}[V]$ of an algebraic set $V \subset \C^{N}$ 
with coordinates $z_1,\cdots,z_N$ is defined as the ring of regular functions on $V$. 
The ring $\mathbf{C}[V]$ is known to be isomorphic to 
the quotient of the polynomial ring $\C[z_1,\cdots,z_N]$ by the ideal $I_V$ 
of all polynomials vanishing on $V$:
\[
\mathbf{C}[V] \cong \C[z_1,\cdots,z_N]/I_V.
\]
Note that $I_V$ is a radical ideal.  
By Hilbert's Nullstellensatz, if $J_V$ is a set of polynomials 
whose common zero locus coincides with $V$, then
\[
I_V=\sqrt{\langle J_V \rangle},
\]
where $\langle * \rangle$ denotes the ideal generated by $*$, 
and $\sqrt{*}$ denotes the radical of an ideal $*$. 
Thus, in order to prove Theorem \ref{defpoly_S0K}, it suffices to determine 
a set $J_{S_0(K)}$ for $V=S_0(K)$, rather than computing $I_{S_0(K)}$ explicitly.
The following theorems provide a topological approach to do this. 
\begin{theorem}[\cite{Bullock, Przytycki-Sikora}]\label{thm_bps}
For a compact orientable $3$-manifold $M$, the map
\[
\varphi : \K_{-1}(M) \to \mathbf{C}[X(\pi_1(M))]
\]
defined by $\varphi(1)=1$ and $\varphi(\gamma)=-t_{\gamma}$ for a loop
$\gamma \in \K_{-1}(M)$
is a surjective $\C$-algebra homomorphism.
Moreover, $\ker(\varphi)$ is the nilradical $\sqrt{0}$.
\end{theorem}
First, Theorem \ref{thm_bps} establishes an isomorphism 
$\K_{-1}(M)/\sqrt{0} \cong \mathbf{C}[X(\pi_1(M))]$,
providing a concrete method to compute the coordinate ring of the character variety 
via Kauffman bracket skein theory.

The next theorem gives an approach for computing the KBSA of a knot exterior.
Let $K$ be a knot given by an $n$-crossing diagram $D_K$.
The knot exterior $E_K$ can be decomposed into a handlebody $H_n$ of genus $n$,
$n$ $2$-handles, and a single $3$-handle (see Figure \ref{fig_decomp_EK}), which corresponds 
to the Wirtinger presentation of $G(K)$ associated with $D_K$. 

\begin{figure}[hbtp]
\[
\begin{minipage}{11cm}
\begin{overpic}[width=\hsize]{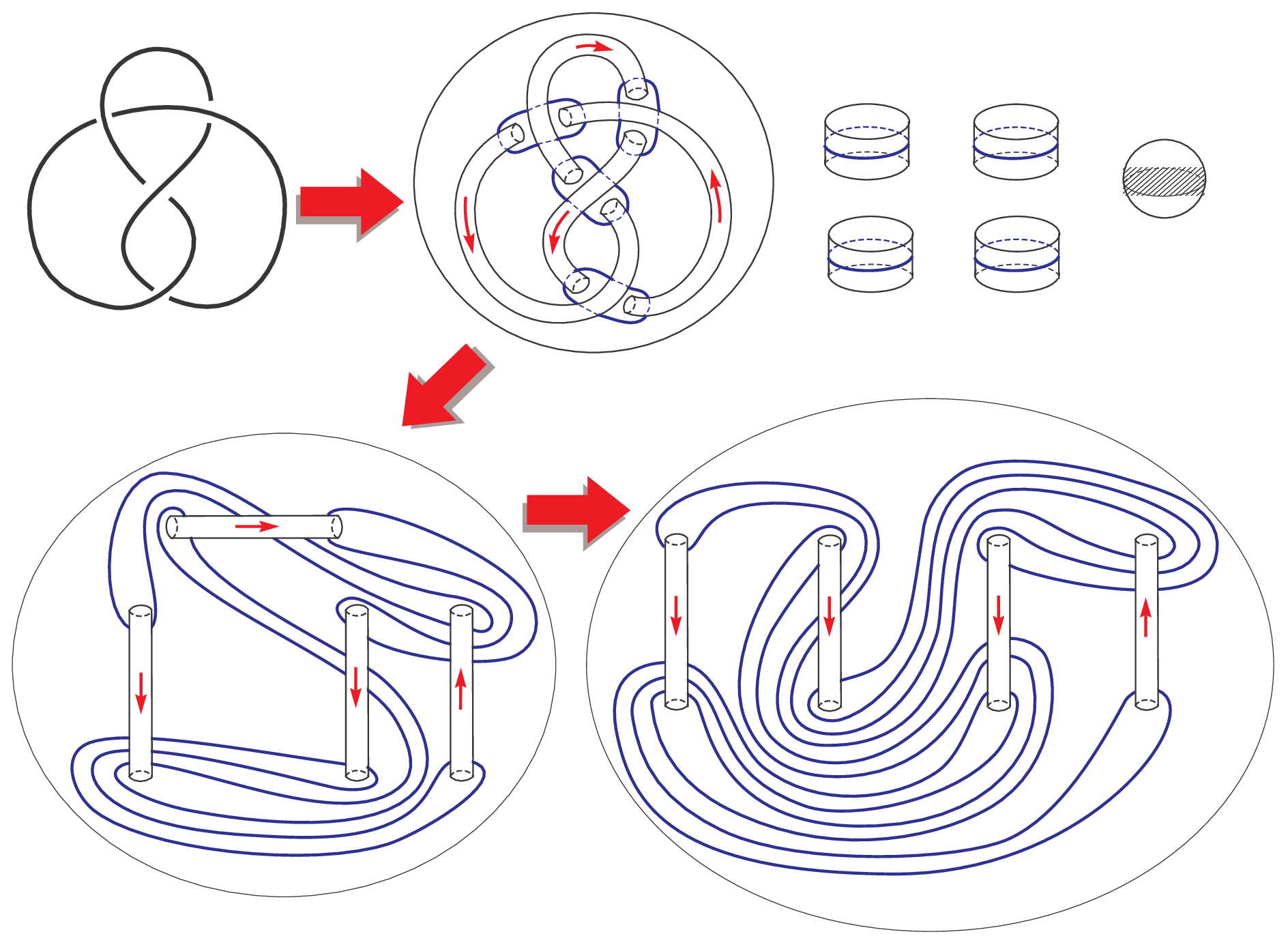}
\put(60.5,61){$\cup$}
\put(72,61){$\cup$}
\put(83,58){$\cup$}
\put(60.5,53){$\cup$}
\put(72,53){$\cup$}
\put(66,46){\small 2-handles}
\put(85,52){\small 3-handle}
\put(6,45){$K=4_1$}
\put(45,42){$H_4$}
\put(33,57){\small $1$}
\put(49.5,69){\small $2$}
\put(40,55){\small $3$}
\put(56,61){\small $4$}
\put(12,18){\small $1$}
\put(19,28){\small $2$}
\put(24,18){\small $3$}
\put(38,18){\small $4$}
\put(54,24){\small $1$}
\put(66,24){\small $2$}
\put(80,24){\small $3$}
\put(91,24){\small $4$}
\end{overpic}
\end{minipage}
\]
\caption{Decomposition of the exterior $E_K$ for the figure-eight knot. 
The attaching curves of the $2$-handles are shown on the boundary of 
the handlebody $H_4$. The meridians of $K$ are right-handed, 
with orientations indicated by arrows on the tubes.}
\label{fig_decomp_EK}
\end{figure}

In this setting,  $\K_{-1}(E_K)$ admits the following presentation. 
\begin{theorem}[\cite{Przytycki1}]\label{thm_van}
There is an isomorphism
\[
\K_{-1}(E_K) \cong
\frac{\K_{-1}(H_n)}
{\left\langle z - sl_b(z) \;\middle|\;
\begin{minipage}{8.5cm}
$z$: any loop in $\K_{-1}(H_n)$,\\
$b$: any band connecting $z$ to an attaching curve
\end{minipage}
\right\rangle}.
\]
Here, $sl_b(z)$ denotes the handle slide of $z$ along the band $b$
connecting $z$ to an attaching curve. 
\end{theorem}
We remark that the right-hand side of Theorem \ref{thm_van} depends on the choice of a diagram $D_K$.  
To make this dependence explicit, we denote by $S_{D_K}$ the ideal that defines 
the quotient on the right-hand side:  
\[
S_{D_K} :=
\left\langle z - sl_b(z) \;\middle|\;
\begin{minipage}{8.5cm}
$z$: any loop in $\K_{-1}(H_n)$,\\
$b$: any band connecting $z$ to an attaching curve
\end{minipage}
\right\rangle.
\]
We refer to $S_{D_K}$ as the \emph{sliding ideal} associated with the diagram $D_K$.

We then describe a procedure to obtain a set $J_{S_0(K)}$ based on the above theorems.  
As shown in Theorem 4.7 of \cite{Przytycki-Sikora2}, the nilradical of $\K_{-1}(H_n)$ is trivial,  
so that the map 
\[
\varphi: \K_{-1}(H_n) \longrightarrow \mathbf{C}[X(F_n)]
\] 
is an isomorphism.  It then follows from Theorems \ref{thm_bps} and \ref{thm_van} that
\begin{eqnarray*}
\mathbf{C}[X(K)] &\cong& \K_{-1}(E_K)/\sqrt{0}\\
&\cong& \K_{-1}(H_n)/\sqrt{S_{D_K}}\\
&\cong& \mathbf{C}[X(F_n)]/\varphi(\sqrt{S_{D_K}})\\
&=& \mathbf{C}[X(F_n)]/\sqrt{\varphi(S_{D_K})}.
\end{eqnarray*}
Let $P=\langle \mathrm{(P1)},\ldots,\mathrm{(P4)}\rangle$.
We denote by 
\[
\psi : \mathbf{C}[X(F_n)] \to \C[x_i;x_{ij};x_{ijk}]/\sqrt{P}
\]
the ring isomorphism defined by $\psi(x)=\bar{x}+\sqrt{P}$ for
$x \in \mathbf{C}[X(F_n)]$, where $\bar{x}$ is obtained from $x$ by replacing
$t_{m_i}$, $t_{m_i m_j}$, and $t_{m_i m_j m_k}$
with $-x_i$, $-x_{ij}$, and $-x_{ijk}$, respectively.
The algebra $\K_{-1}(H_n)$ is finitely generated (see \cite{Bullock2}) and Noetherian 
by the isomorphism $\varphi$. 
Let $u_1,\ldots,u_m$ be skeins in $\K_{-1}(H_n)$ generating the ideal $S_{D_K}$.
With this setup, we obtain the following correspondence:
\begin{eqnarray*}
\mathbf{C}[X(K)]  
&\cong& \mathbf{C}[X(F_n)]/\sqrt{\langle \varphi(u_1),\cdots, \varphi(u_m) \rangle}\\
&\cong & (\C[x_i;x_{ij};x_{ijk}]/\sqrt{P})
/\psi(\sqrt{\langle \varphi(u_1),\cdots,\varphi(u_m) \rangle})\\
& = & (\C[x_i;x_{ij};x_{ijk}]/\sqrt{P})
/\sqrt{\psi(\langle \varphi(u_1),\cdots, \varphi(u_m) \rangle)}\\
& = & (\C[x_i;x_{ij};x_{ijk}]/\sqrt{P})
/\sqrt{\langle \overline{\varphi(u_1)}+\sqrt{P},\cdots, \overline{\varphi(u_m)}
+\sqrt{P} \rangle}\\
&\cong& 
\C[x_i;x_{ij};x_{ijk}]
/\sqrt{\langle \overline{\varphi(u_1)},\cdots, \overline{\varphi(u_m)} \rangle +\sqrt{P}}\\
&=& \C[x_i;x_{ij};x_{ijk}]
/\sqrt{\langle \overline{\varphi(u_1)},\cdots, \overline{\varphi(u_m)} \rangle + P}. 
\end{eqnarray*}
The last equality follows from $\sqrt{\sqrt{I}+\sqrt{J}}=\sqrt{I+J}$ for ideals $I$ and $J$. 
The above argument shows that, for $X(K) \subset \C^{n+{n \choose 2} + {n \choose 3}}$,  
\[
I_{X(K)}=\sqrt{\langle \overline{\varphi(u_1)},\cdots, \overline{\varphi(u_m)} \rangle + P} 
\subset \C[x_i;x_{ij};x_{ijk}].
\]
Since the algebraic set $\tilde{t}(\mathfrak{S}_0(K))$ is the intersection of
$X(K)$ with the hyperplane defined by $x_i=0$ $(i=1,\cdots,n)$,
we obtain the following description of the ideal
$I_{\tilde{t}(\mathfrak{S}_0(K))} \subset \C[x_i;x_{ij};x_{ijk}]$:
\begin{eqnarray*}
I_{\tilde{t}(\mathfrak{S}_0(K))}
&=& \sqrt{I_{X(K)} + \langle x_1,\cdots,x_n\rangle}\\
&=& \sqrt{ \langle \overline{\varphi(u_1)},\cdots, \overline{\varphi(u_m)} \rangle
+ P + \langle x_1,\cdots,x_n\rangle}\\
&=& \sqrt{ \langle \overline{\varphi(u_1)},\cdots, \overline{\varphi(u_m)} \rangle
+ \langle {\rm (GH)} \rangle + \langle x_1,\cdots,x_n\rangle}
\end{eqnarray*}
Finally, the ideal $I_{S_0(K)}$ of the algebraic set
$S_0(K)=p(\tilde{t}(\mathfrak{S}_0(K))) \subset \C^{{n \choose 2}+{n \choose 3}}$,
where $p$ is the projection defined right before Theorem $\ref{defpoly_S0K}$, 
is obtained as the elimination ideal of
$I_{\tilde{t}(\mathfrak{S}_0(K))}$.
More precisely, the ideal $I_{S_0(K)} \subset \C[x_{ij};x_{ijk}]$
can be computed by the following basic arguments:
\begin{eqnarray*}
I_{S_0(K)}&=& I_{\tilde{t}(\mathfrak{S}_0(K))} \cap \C[x_{ij};x_{ijk}]\\
&=& \sqrt{ \langle \overline{\varphi(u_1)},\cdots, \overline{\varphi(u_m)} \rangle
+ \langle \mathrm{(GH)} \rangle + \langle x_1,\cdots, x_n \rangle}
\cap \C[x_{ij};x_{ijk}]\\
&=& \sqrt{ (\langle \overline{\varphi(u_1)},\cdots, \overline{\varphi(u_m)} \rangle
+ \langle \mathrm{(GH)} \rangle + \langle x_1,\cdots, x_n \rangle)
\cap \C[x_{ij};x_{ijk}]}
\end{eqnarray*}
Hence, the generators of the elimination ideal
\[
(\langle \overline{\varphi(u_1)},\cdots, \overline{\varphi(u_m)} \rangle
+ \langle \mathrm{(GH)} \rangle + \langle x_1,\cdots, x_n \rangle)
\cap \C[x_{ij};x_{ijk}]
\]
determine the defining polynomials of $S_0(K)$. 

We can compute the generators of the above elimination ideal effectively 
under the trace-free condition on $\K_{-1}(H_n)$, as follows. 
Let $m_1,\cdots,m_n$ be the standard generators\footnote{We use the letter $m$ 
since these generators correspond to meridional loops in $\K_{-1}(E_K)$.}
of $\pi_1(H_n)$, and let $s_{i_1\cdots i_k}$ denote the skein in $\K_{-1}(H_n)$
represented by a loop freely homotopic to $m_{i_1}\cdots m_{i_k}$
(see Figure \ref{skein_s}).
With this notation, $\K_{-1}(H_n)$ can be regarded as a $\C$-algebra generated by
$s_i, s_{ij}, s_{ijk}$ (compare with $\mathbf{C}[X(F_n)]$).

\begin{figure}[hbtp]
\[
\begin{minipage}{4.5cm}
\begin{overpic}[width=\hsize]{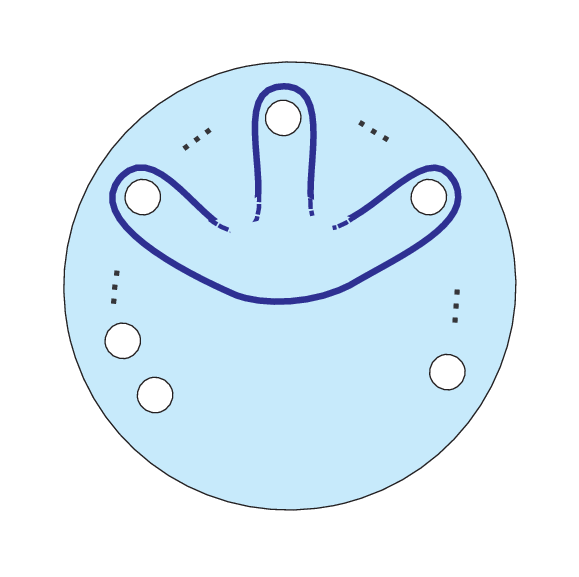}
\put(12,75){$i_1$}
\put(45,95){$i_j$}
\put(83,75){$i_k$}
\put(5,35){$2$}
\put(13,20){$1$}
\put(87,30){$n$}
\end{overpic}
\end{minipage}
\vspace*{-1cm}
\]
\caption{Skein $s_{i_1 \dots i_k}$ for $1 \leq i_1, \cdots, i_k \leq n$ 
in the handlebody $H_n \cong D_n \times [0,1]$, regarded as 
a planar curve on $D_n \times \{0\}$. Every meridian $m_i$ is oriented 
clockwise. }
\label{skein_s}
\end{figure}

For an ideal $I$ in $\mathbf{C}[X(F_n)]$, let $\bar{I}$ denote the ideal in $\C[x_i;x_{ij};x_{ijk}]$ 
obtained by substituting
$-x_i$, $-x_{ij}$, and $-x_{ijk}$ for $t_{m_i}$, $t_{m_i m_j}$, and $t_{m_i m_j m_k}$, respectively.
Then, using the isomorphisms $\varphi$ and $\psi$ together with standard arguments, we obtain
\[
\langle \overline{\varphi(u_1)},\cdots, \overline{\varphi(u_m)} \rangle
+ \langle x_1,\cdots,x_n \rangle +\langle {\rm (GH)} \rangle
=
\overline{\varphi\left(S_{D_K} + \langle s_1,\cdots,s_n \rangle\right)}  +\langle {\rm (GH)} \rangle.
\]
The ideal $S_{D_K}$ for $S_{D_K} + \langle s_1,\dots,s_n \rangle \subset \K_{-1}(H_n)$ 
can be regarded as $S_{D_K}$ with the relations $s_1=\cdots=s_n=0$, 
which we call the trace-free condition in $\K_{-1}(H_n)$. 
This viewpoint significantly simplifies the computation of the ideal 
$S_{D_K} + \langle s_1,\dots,s_n \rangle$, 
and consequently that of the elimination ideal.

We now describe the structure of the sliding ideal $S_{D_K}$ under the trace-free condition.
Let $\K_{-1,\mathrm{TF}}(H_n)$ denote the Kauffman bracket skein algebra $\K_{-1}(H_n)$
specialized to the trace-free condition, and let $S_{D_K}|_{s_i=0}$ denote the corresponding
specialization of $S_{D_K}$.
\begin{lemma}\label{lem_sdk}
The ideal $S_{D_K}|_{s_i=0}$ coincides with the ideal $F_{D_K}$
in $\K_{-1,\mathrm{TF}}(H_n)$ defined by
\begin{eqnarray*}
F_{D_K}
:=\left\langle 
\left.
\begin{array}{l}
s_{ak}-s_{ij}s_{ai}+s_{aj} \ \text{{\rm : (F2)}},\\ 
s_{bck}-s_{ij}s_{bci}+s_{bcj} \ \text{{\rm : (F3)}}
\end{array}
\right| \begin{array}{l}
\mbox{$(i,j,k)$ {\rm : any Wirtinger triple for $D_K$}},\\
1 \leq a \leq n,\ 1 \leq b < c \leq n
\end{array}
\right\rangle.
\end{eqnarray*}
\end{lemma}
We refer to $F_{D_K}$ as the \emph{fundamental ideal} associated with $D_K$.
In the proof of Lemma \ref{lem_sdk}, we show that the trace-free specialization 
$S_{D_K}|_{s_i=0}$ is generated by finitely many ``non-winding band sums'', 
and these band sums are generated by (F2) and (F3) in $F_{D_K}$.

To compute a generating set for $S_{D_K}$, we isotope the handlebody $H_n$
into the product $D_n \times [0,1]$, where $D_n$ is an $n$-punctured disk.
Under this isotopy, the attaching curves of the $2$-handles 
can be regarded as curves on the boundary of $D_n \times [0,1]$. 
We then project these attaching curves onto $D_n \times \{0\}$.

\begin{figure}[hbtp]
\[
\begin{minipage}{11cm}
\begin{overpic}[width=\hsize]{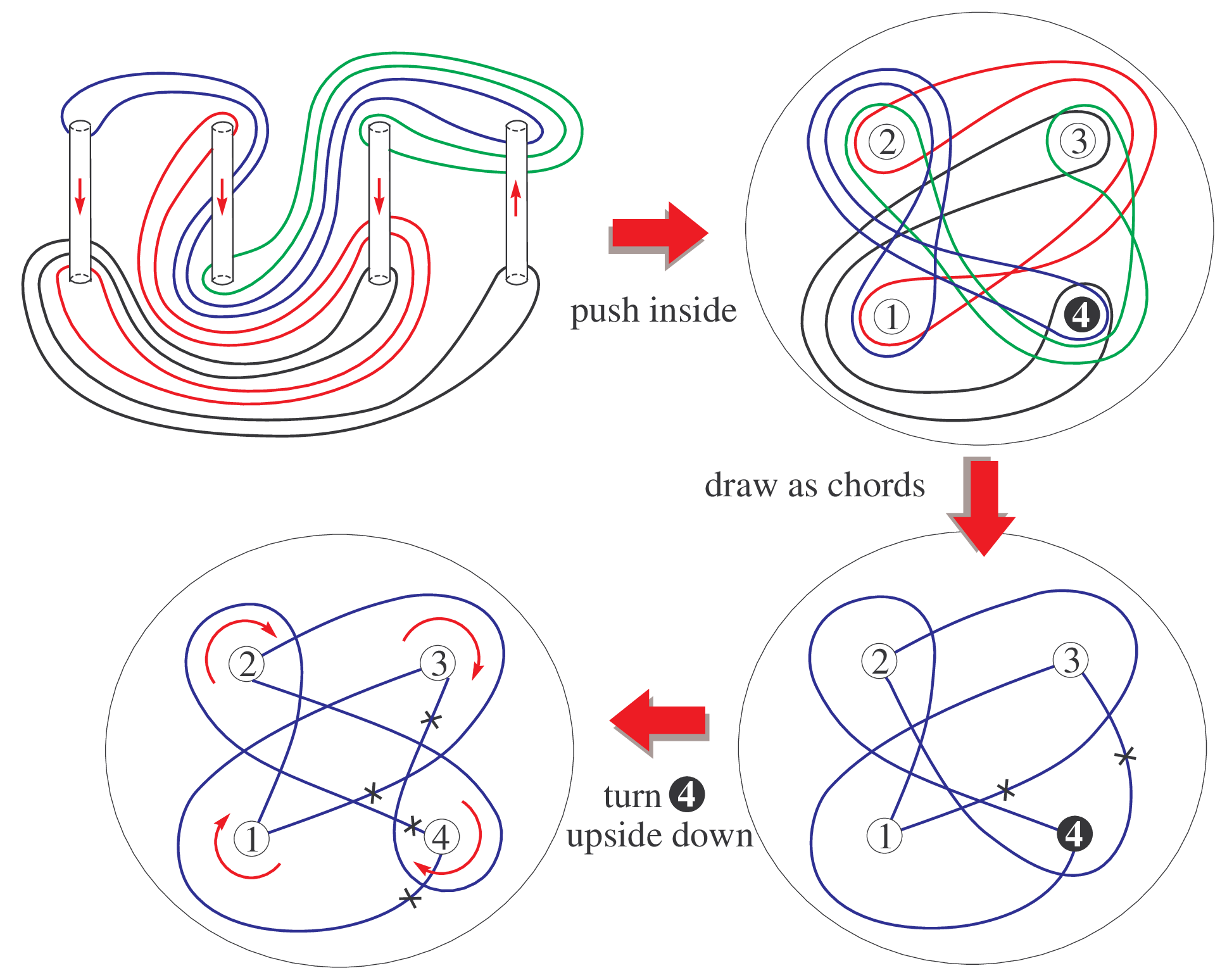}
\put(9,63){\small $1$}
\put(21,63){\small $2$}
\put(33,63){\small $3$}
\put(44,63){\small $4$}
\end{overpic}
\end{minipage}
\]
\caption{Attaching curves and chords. Each of the cross signs ``$\times$" 
on chords presents a crossing of the attaching curve where the sign is.}
\label{attaching_curve}
\end{figure}

For simplicity, in Figure \ref{attaching_curve}, they are drawn as chords 
on $D_n \times \{0\}$.
These projected curves on $D_n \times \{0\}$ represent the relations in the Wirtinger presentation 
of the knot group $G(K)$ associated with $D_K$.  
Thus, in the projection process, the attaching curves are considered up to homotopy. 

As seen in Figure \ref{attaching_curve}, any attaching curve on $D_n \times \{0\}$ 
falls into one of the following two types:
\[
\begin{minipage}{2.7cm}\includegraphics[width=\hsize]{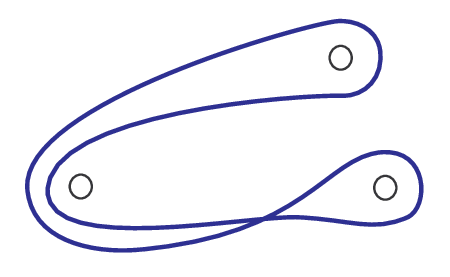}\end{minipage},
\hspace*{0.5cm}
\begin{minipage}{3cm}\includegraphics[width=\hsize]{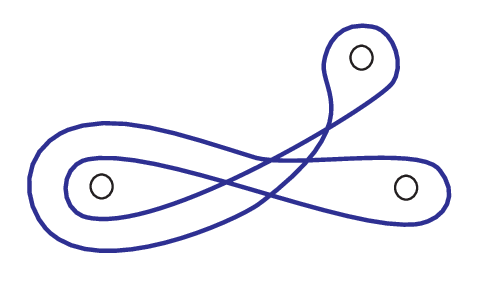}\end{minipage}.
\]
This result also follows from the relation $m_k=m_im_jm_i^{-1}$ in the Wirtinger presentation. 
In the following, we prove Lemma \ref{lem_sdk} for the first type, 
because the argument for the second type is entirely analogous.

\fbox{\bf Step 1}
We begin with the original sliding ideal $S_{D_K}$ in $\K_{-1}(H_n)$.
Let $z \in \K_{-1}(H_n)$ be a loop, and let $b$ be a band connecting $z$
to an attaching curve. Performing the handle slide of $z$ along $b$, we obtain
\[
z-sl_b(z)
=z-
\begin{minipage}{2.5cm}
\begin{overpic}[width=\hsize]{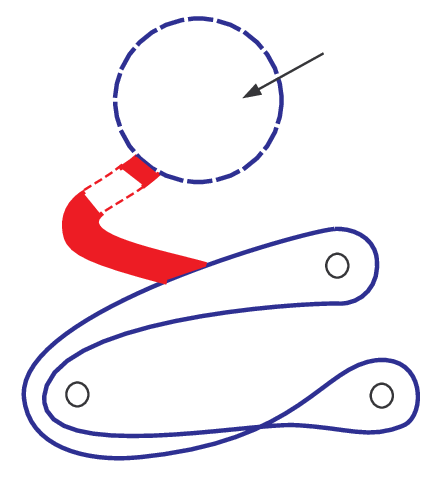}
\put(70,90){\footnotesize resolve $z$ by the skein relations}
\put(36,77){$z$}
\put(6,62){$b$}
\end{overpic}
\end{minipage}, 
\]
where the band $b$ is partially omitted, as indicated by the blank.
This description of a handle slide is referred to as a \emph{band sum}.

In what follows, it suffices to consider only \emph{untwisted} band sums.
Indeed, let $b_{\times}$ be a band with a single twist connecting $z$ to 
an attaching curve $c$. Then the expression $z-sl_{b_{\times}}(z)$ can be reduced to
\[
-(z-sl_{b_o}(z)) + z\,(c-(-2)),
\]
where $b_o$ denotes the band obtained from $b_{\times}$ by smoothing the twist.
Note that $c-(-2)$ can be written as $c-sl_{b'}(c)$ for an untwisted band $b'$ 
connecting $c$ to a parallel copy of itself.

We first resolve the loop $z$ to which the band $b$ is attached using the skein relations. 
Since $\K_{-1}(H_n) \cong \mathbf{C}[X(F_n)]$, the loop $z$ can be expressed as  
\[
z=\fbox{$\bigcirc$} f+\sum_{1 \leq i \leq n}\fbox{$s_i$}f_i
+\sum_{1 \leq i<j \leq n}\fbox{$s_{ij}$}f_{ij}+\sum_{1 \leq i<j<k\leq n}\fbox{$s_{ijk}$}f_{ijk}, 
\]
where $f, f_i,f_{ij}$, and $f_{ijk}$ are polynomials in the skeins $s_i,s_{ij},s_{ijk}$. 
Each boxed term in the above expression
indicates the loop to which the band $b$ will be attached. 
For example, these loops are determined by the skein relations as follows:
\begin{eqnarray}\label{bandsum}
\begin{minipage}{11cm}
\begin{overpic}[width=\hsize]{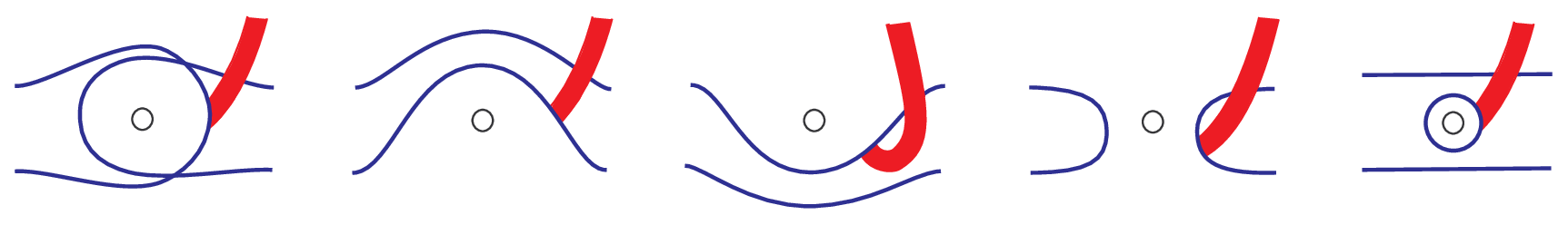}
\put(18,6){$=$}
\put(39,6){$+$}
\put(61,6){$+$}
\put(83,6){$+$}
\end{overpic}
\end{minipage}. 
\end{eqnarray}
Thus, by continuing this resolution until every loop is expressed 
in terms of $\bigcirc=-2, s_i, s_{ij}$ or $s_{ijk}$, we obtain the above expression. 
Applying the same procedure to $sl_b(z)$, we find that 
$sl_b(z)$ can be written as  
\[
\fbox{
\begin{minipage}{2.5cm}
\begin{overpic}[width=\hsize]{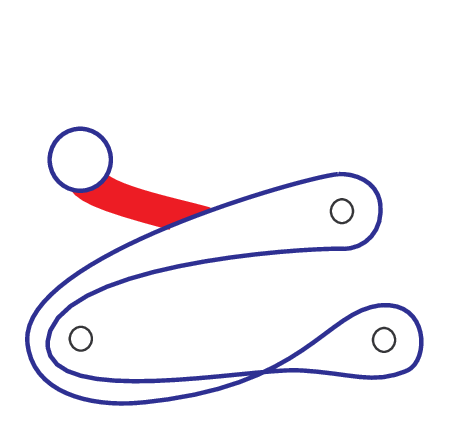}
\put(35,60){\small $b_0$}
\end{overpic}
\end{minipage}}f
+\Sigma_{i}
\fbox{\begin{minipage}{2.5cm}
\begin{overpic}[width=\hsize]{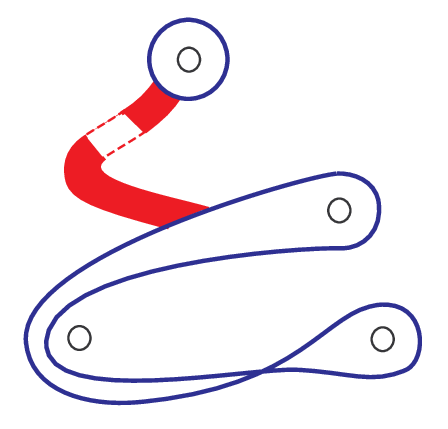}
\put(5,70){\small $b_1$}
\put(55,80){\small $i$}
\end{overpic}
\end{minipage}}f_i
+\Sigma_{i,j}
\fbox{\begin{minipage}{2.5cm}
\begin{overpic}[width=\hsize]{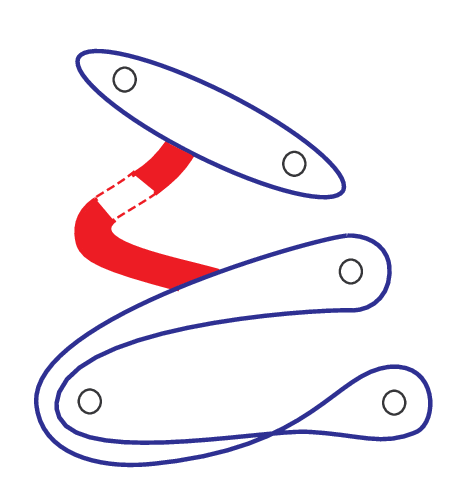}
\put(5,65){\small $b_2$}
\put(25,90){\small $i$}
\put(70,72){\small $j$}
\end{overpic}
\end{minipage}}f_{ij}
+\Sigma_{i,j,k}
\fbox{\begin{minipage}{2.5cm}
\begin{overpic}[width=\hsize]{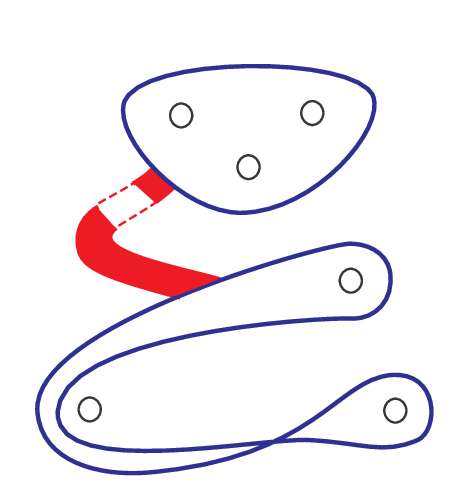}
\put(5,65){\small $b_3$}
\put(32,90){\small $i$}
\put(65,90){\small $j$}
\put(43,47){\small $k$}
\end{overpic}
\end{minipage}}f_{ijk}.
\]
It follows that any handle slide is expressed by 
\[
(\bigcirc-sl_{b_0}(\bigcirc))f+\Sigma_i(s_i-sl_{b_1}(s_i))f_i
+\Sigma_{i,j}(s_{ij}-sl_{b_2}(s_{ij}))f_{ij}+\Sigma_{i,j,k}(s_{ijk}-sl_{b_3}(s_{ijk}))f_{ijk}.
\]
and hence 
\[
S_{D_K}=\left\langle \bigcirc-sl_b(\bigcirc),\ s_i-sl_b(s_i),\ s_{ij}-sl_b(s_{ij}),\ 
s_{ijk}-sl_b(s_{ijk})\mid \mbox{\small $b$: any band} \right\rangle. 
\]

\fbox{\bf Step 2}
We consider $sl_b(s)$ for a skein $s$ in $\{\bigcirc=-2,\ s_i,\ s_{ij},\ s_{ijk}\}$.
If the band $b$ is winding, that is, if it goes around at least one puncture,
then $sl_b(s)$ can be expressed as a sum of band sums associated with
non-winding bands, with coefficients in $\K_{-1}(H_n)$,
by the skein relations (\ref{bandsum}).
Indeed, if $b$ is a winding band that goes once around a puncture,
then $sl_b(s)$ can be written as
\begin{eqnarray*}
\begin{minipage}{2.5cm}
\begin{overpic}[width=\hsize]{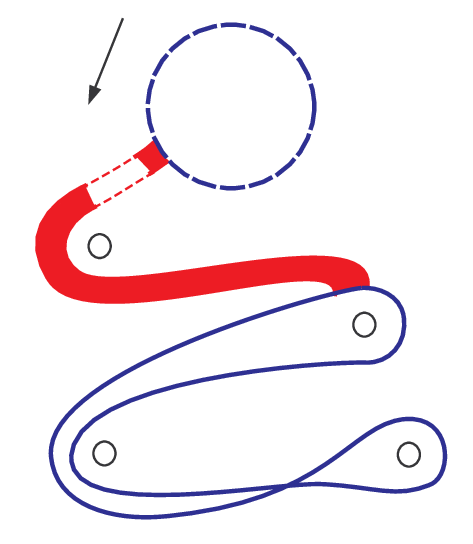}
\put(10,70){\small $b$}
\put(0,99){\footnotesize a winding band}
\put(25,55){\small $a$}
\put(38,77){$s$}
\end{overpic}
\end{minipage}
\hspace*{-0.5cm}
&=&
\begin{minipage}{2.5cm}
\begin{overpic}[width=\hsize]{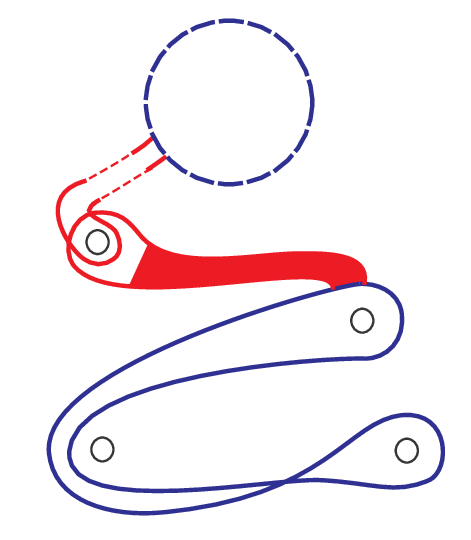}
\put(0,50){\small $a$}
\put(38,77){$s$}
\put(10,70){$c$}
\put(60,58){\small $b_1$}
\end{overpic}
\end{minipage}
-
\begin{minipage}{2.5cm}
\begin{overpic}[width=\hsize]{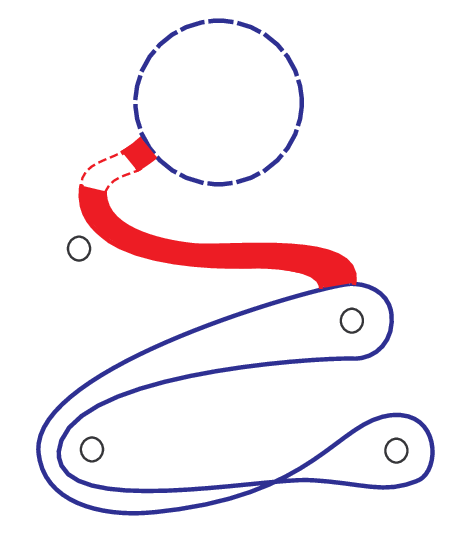}
\put(0,50){\small $a$}
\put(36,78){$s$}
\put(60,58){\small $b_2$}
\end{overpic}
\end{minipage}
-
\begin{minipage}{2.5cm}
\begin{overpic}[width=\hsize]{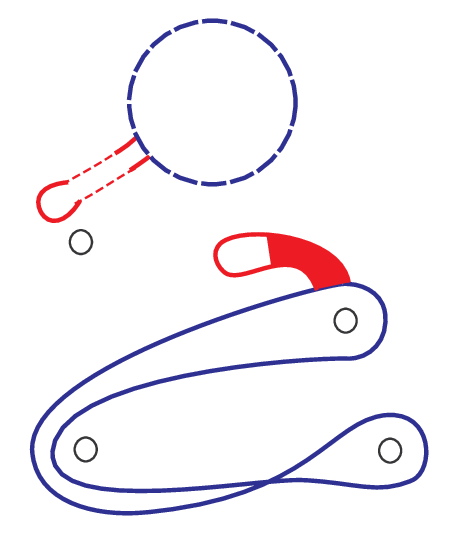}
\put(2,50){\small $a$}
\put(35,78){$s$}
\put(60,58){\small $b_3$}
\end{overpic}
\end{minipage}
-
\begin{minipage}{2.5cm}
\begin{overpic}[width=\hsize]{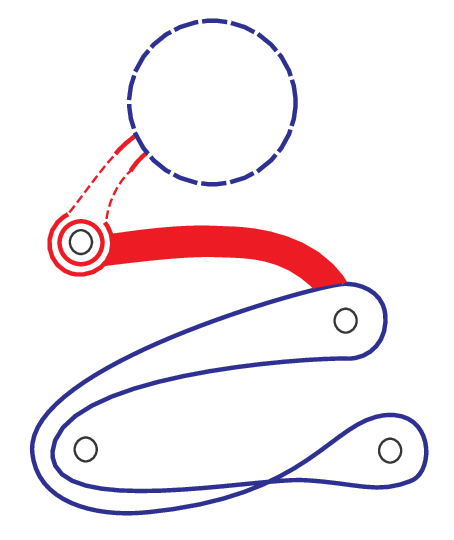}
\put(0,50){\small $a$}
\put(10,70){$c$}
\put(35,78){$s$}
\put(60,58){\small $b_4$}
\end{overpic}
\end{minipage}\\
&=&(s \sharp s_a) \cdot sl_{b_1}(s_a)-sl_{b_2}(s)-s \cdot sl_{b_3}(\bigcirc)
-s_a \cdot sl_{b_4}(s \sharp s_a), 
\end{eqnarray*}
where $s \sharp s_a$ denotes the band sum of $s$ and $s_a$ 
along a non-winding band $c$. Consequently, $s - sl_b(s)$ can be expressed as
\[
(s \sharp s_a) \cdot 
\left(s_a-sl_{b_1}(s_a)\right)
-\left(s-sl_{b_2}(s)\right)
-s \cdot \left(-2-sl_{b_3}(\bigcirc)\right)
-s_a \cdot \left(s \sharp s_a-sl_{b_4}(s \sharp s_a)\right)  
\]
where each $b_i$ is a non-winding band. 

We now show that $S_{D_K}$ is generated by
\begin{eqnarray}\label{gen-sdk}
-2-sl_{*}(\bigcirc),\ s_i-sl_{*}(s_i),\ s_{ij}-sl_{*}(s_{ij}),\ s_{ijk}-sl_{*}(s_{ijk}),
\end{eqnarray}
where $*$ ranges over all non-winding bands.
The claim is immediate when $s=\bigcirc$, $s_i$, or $s_{ij}$.
Suppose that $s=s_{ijk}$.
Then a band sum $s \sharp s_a$ may yield the skein $s_{aijk}$.
By applying the skein relation (\ref{bandsum}), the skein $s \sharp s_a$  
reduces to a polynomial $p$ in $\bigcirc$, $s_i$, $s_{ij}$, and $s_{ijk}$, 
with four boxed terms \fbox{$\bigcirc$}, \fbox{$s_i$}, \fbox{$s_{ij}$}, and \fbox{$s_{ijk}$}. 
On the other hand, the term $sl_{b_4}(s \sharp s_a)$ corresponds to the band sum 
$sl_{b_4}(s_{aijk})$ along a non-winding band $b_4$. 
Applying the skein relation (\ref{bandsum}) to this band sum, 
we see that $sl_{b_4}(s_{aijk})$ is expressed as the polynomial obtained from $p$ 
by replacing each boxed term with its corresponding non-winding band sum
arising from (\ref{bandsum}). This implies that $s \sharp s_a-sl_{b_4}(s \sharp s_a)$ 
is generated by the skeins in (\ref{gen-sdk}), which proves the claim in this case. 

By induction, for any band $b$, continuing the above process until all winding bands are eliminated, 
we obtain the following description of $s-sl_{b}(s)$: 
\[
\left(-2-sl_{*}(\bigcirc)\right)f
+\Sigma_i \left(s_i-sl_{*}(s_i)\right)f_i
+\Sigma_{ij} \left(s_{ij}-sl_{*}(s_{ij})\right)f_{ij}
+\Sigma_{ijk} \left(s_{ijk}-sl_{*}(s_{ijk})\right)f_{ijk}, 
\]
where $f,f_i,f_{ij},f_{ijk} \in \K_{-1}(H_n)$ and $*$ denotes unspecified non-winding bands. 
Hence, 
\[
S_{D_K}=\left\langle 
\bigcirc-sl_{*}(\bigcirc),s_i-sl_{*}(s_i),s_{ij}-sl_{*}(s_{ij}),s_{ijk}-sl_{*}(s_{ijk}) \mid 
\mbox{\small $*:$ any non-winding band} \right\rangle.
\]
Note that for a skein $s$ in $\left\{\bigcirc=-2,\ s_i,\ s_{ij},\ s_{ijk}\right\}$ there exist 
only finitely many non-winding bands (up to homotopy) connecting $s$ to an attaching curve. 
Therefore the above generating set of $S_{D_K}$ determines the generators $u_1,\cdots,u_m$.  

\fbox{\bf Step 3} 
From now on, we work with the specialization $S_{D_K}|_{s_i=0}$. 
It follows from (\ref{bandsum}) that the skein relations under the trace-free condition 
$s_1=\cdots=s_n=0$ (which we call the trace-free skein relations) are given as follows: 
\begin{eqnarray}\label{ff-formula}
\begin{minipage}{9cm}\includegraphics[width=\hsize]{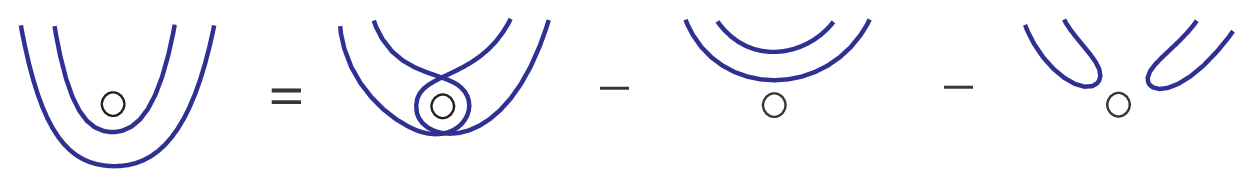}\end{minipage}. 
\end{eqnarray}
If $sl_b(s_*)$ is a non-winding band sum of 
$s_*\in \left\{\bigcirc=-2,\ s_i=0,\ s_{ij},\ s_{ijk}\right\}$ 
and an attaching curve disjoint from $s_*$, then by (\ref{ff-formula})  
we obtain the following expression for any Wirtinger triple $(p,q,r)$: 
\[
s_*-sl_b(s_*)
=s_*
-\begin{minipage}{3cm}
\begin{overpic}[width=\hsize]{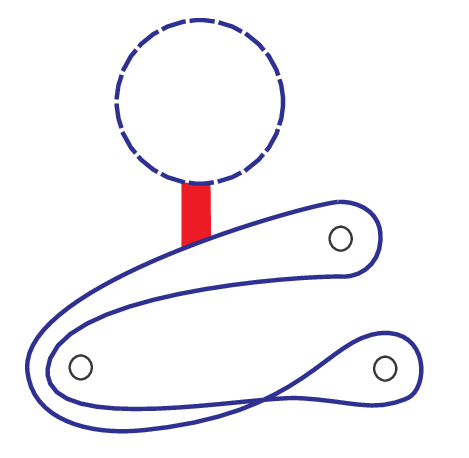}
\put(27,48){\small $b$}
\put(38,72){$s_*$}
\put(25,18){$p$}
\put(95,18){$q$}
\put(88,45){$r$}
\end{overpic}
\end{minipage}
=s_*+(s_* \sharp s_{pr})s_{pq}-(s_* \sharp s_{qr}). 
\]
Since $s_* \sharp s_{pr}$ and $s_* \sharp s_{qr}$ are non-winding band sums (without twists), 
they can be expressed as polynomials in $s_{ij}$ and $s_{ijk}$ 
using the trace-free skein relations (\ref{ff-formula}). 
Consequently, the resulting polynomial lies in the ideal $F_{D_K}$, which is generated by (F2) and (F3) 
in $F_{D_K}$. 

To be more precise, if $s_*=s_i=0$, then for $i<p<q<r$ we have
\[
s_i-sl_b(s_i)
=(s_i \sharp s_{pr})\, s_{pq}- (s_i \sharp s_{qr})= s_{ipr} s_{pq}-s_{iqr}=s_{irr}-s_{pq}s_{irp}+s_{irq}, 
\]
which is one of (F3) in $F_{D_K}$. The remaining cases can be shown in a similar manner. 

When $s_*=s_{ij}$ and $i<j<p<q<r$, the trace-free skein relations give 
\begin{eqnarray*}
&&s_{ij} \sharp s_{pr} = s_{ijpr}
=\frac{1}{2}\bigl(s_{ip}s_{jr}-s_{ij}s_{pr}-s_{ir}s_{jp}\bigr),\\
&&s_{ij} \sharp s_{qr} = s_{ijqr}
=\frac{1}{2}\bigl(s_{iq}s_{jr}-s_{ij}s_{qr}-s_{ir}s_{jq}\bigr).
\end{eqnarray*}
Hence, $s_{ij}-sl_b(s_{ij})=s_{ij}+(s_{ij} \sharp s_{pr})s_{pq}-(s_{ij} \sharp s_{qr})$ is expressed by 
\[
-\frac{1}{2}s_{jr}(s_{ir}-s_{ip}s_{pq}+s_{iq})
+\frac{1}{2}s_{ij}(2-s_{pr}s_{pq}+s_{qr})
+\frac{1}{2}s_{ir}(s_{jr}-s_{jp}s_{pq}+s_{jq}),
\]
which is generated by (F2) in $F_{D_K}$. 
The remaining cases can be shown analogously.

For the case $s_*=s_{ijk}$, we use the relation 
\[
s_{abcde}
=\frac{1}{2}\left(
s_{be}s_{acd}-s_{ae}s_{bcd}-s_{ab}s_{cde}-s_{cd}s_{abe}
\right)
\ \ (1 \leq a < b < c < d < e \leq n), 
\]
which is given by the trace-free skein relations, 
to resolve $s_{ijk} \sharp s_{pr}$ and $s_{ijk} \sharp s_{qr}$.
For $i < j < k < p < q < r$, we obtain
\begin{eqnarray*}
s_{ijk}+(s_{ijk} \sharp s_{pr})s_{pq}-(s_{ijk} \sharp s_{qr})
&=&
-\frac{1}{2}s_{jr}(s_{ikr}-s_{pq}s_{ikp}+s_{ikq})
+\frac{1}{2}s_{ir}(s_{jkr}-s_{pq}s_{jkp}+s_{jkq})\\
&&
+\frac{1}{2}s_{ij}(s_{krr}-s_{pq}s_{kpr}+s_{kqr})
+\frac{1}{2}s_{ijr}(s_{kr}-s_{pq}s_{kp}+s_{kq}),
\end{eqnarray*}
which is generated by (F2) and (F3) in $F_{D_K}$. 
Here we have applied the identity
\[
s_{ijk}
=-s_{ijkrr}
=-\frac{1}{2}\left(
s_{jr}s_{ikr}-s_{ir}s_{jkr}-s_{ij}s_{krr}-s_{kr}s_{ijr}
\right).
\]

The same argument as above applies when $s_*$ intersects an attaching curve,
showing that $s_*-sl_b(s_*)$ still lies in $F_{D_K}$.
We therefore omit the details.
Hence, all generators of $S_{D_K}|_{s_i=0}$ can be reduced to (F2) and (F3) in $F_{D_K}$, 
and consequently
\[
S_{D_K}|_{s_i=0}=F_{D_K}, 
\]
proving Lemma \ref{lem_sdk}.

We now complete the proof of Theorem \ref{defpoly_S0K}.
By the equality $S_{D_K}|_{s_i=0}=F_{D_K}$, we obtain 
\[
I_{S_0(K)}
=\sqrt{\langle x_{ak}-x_{ij}x_{ai}+x_{aj}, x_{bck}-x_{ij}x_{bci}+x_{bcj}, {\rm (GH)} \rangle} 
\subset \C[x_{ij};x_{ijk}].
\]
This shows that the common zeros locus of 
\[
x_{ak}-x_{ij}x_{ai}+x_{aj}=\overline{\varphi({\rm (F2)})},\   
x_{bck}-x_{ij}x_{bci}+x_{bcj}=\overline{\varphi({\rm (F3)})},\ 
{\rm (GH)} 
\]
coincide with the algebraic set $S_0(K)$ in $\C^{{n \choose 2}+{n \choose 3}}$. 
For simplicity, we write $\overline{\varphi({\rm (F2)})}$ and 
$\overline{\varphi({\rm (F3)})}$ as (F2) and (F3), respectively. 
We call them the fundamental relations for $S_0(K)$. 

In fact, the relations (F3) can be eliminated using (F2) and (GH). 
Indeed, any point $(x_{ab};x_{pqr})$ satisfying (F2) and (GH)
must also satisfy (F3): $x_{abk}=x_{ij}x_{abi}-x_{abj}$ for each Wirtinger triple $(i,j,k)$.
For instance, if $x_{pqr}=0$ for all $(p,q,r)$, then (F3) holds trivially.
Otherwise, suppose that $x_{stu}\neq 0$ for some triple $(s,t,u)$. Then we have
\begin{eqnarray*}
x_{stu}x_{abk} &=& \frac{1}{2}
\left|
\begin{array}{ccc}
x_{sa} & x_{sb} & x_{sk}\\
x_{ta} & x_{tb} & x_{tk}\\
x_{ua} & x_{ub} & x_{uk}
\end{array}
\right|
= \frac{1}{2}
\left|
\begin{array}{ccc}
x_{sa} & x_{sb} & x_{ij}x_{si}-x_{sj}\\
x_{ta} & x_{tb} & x_{ij}x_{ti}-x_{tj}\\
x_{ua} & x_{ub} & x_{ij}x_{ui}-x_{uj}
\end{array}
\right|\\
&=& \frac{1}{2}x_{ij}
\left|
\begin{array}{ccc}
x_{sa} & x_{sb} & x_{si}\\
x_{ta} & x_{tb} & x_{ti}\\
x_{ua} & x_{ub} & x_{ui}
\end{array}
\right|
-\frac{1}{2}
\left|
\begin{array}{ccc}
x_{sa} & x_{sb} & x_{sj}\\
x_{ta} & x_{tb} & x_{tj}\\
x_{ua} & x_{ub} & x_{uj}
\end{array}
\right|\\
&=& x_{ij}x_{stu}x_{abi}-x_{stu}x_{abj}\\
&=& x_{stu}(x_{ij}x_{abi}-x_{abj}). 
\end{eqnarray*}
Since $x_{stu}\neq 0$, we obtain $x_{abk}=x_{ij}x_{abi}-x_{abj}$. 
This completes the proof of Theorem \ref{defpoly_S0K}.


\subsection{Trace-free slice and fundamental variety}\label{sec_obs}
We will observe how the relations (F2) and (GH) work to determine the trace-free slice $S_0(K)$. 
In the following, for a knot $K$ given by an $n$-crossing diagram $D_K$, 
we denote by $F_2(K)$ the common zero locus of (F2) associated with $D_K$: 
\[
F_2(K):=\left\{(x_{12},\cdots,x_{nn-1}) \in \C^{n \choose 2} 
\left|\ 
\begin{array}{l}
x_{ak}-x_{ij}x_{ai}+x_{aj}=0\ :\ {\rm (F2)}\\
\mbox{\small $(i,j,k)$: any Wirtinger triple of $D_K$, $a \in \{1,\cdots,n\}$}, 
\end{array}
\right.\right\},
\]
which is referred to as the fundamental variety of $K$. 
Note that $F_2(K)$ does not depend on the choice of a diagram $D_K$ 
up to biregular equivalence, which will be shown in the relationship of $F_2(K)$ 
to degree 0 abelian knot contact homology (see Proposition \ref{contact-F2K}). 

Since the trefoil knot is too simple to exhibit the phenomena of interest,
we first consider the figure-eight knot $4_1$, using the diagram and arc labels
shown in Figure \ref{fig_decomp_EK}, where the arc labels are indicated on the tubes.
In this setting, there exist four Wirtinger triples $(1,3,4),(2,1,4),(3,1,2)$, and $(4,2,3)$. 
By Theorem \ref{defpoly_S0K}, these yield the following 16 fundamental relations (F2): 
\[
\begin{array}{llll}
x_{14}=x_{13}, & x_{24}=x_{13}x_{12}-x_{23}, & x_{34}=x_{13}^2-2, & 2=x_{13}x_{14}-x_{34},\\
x_{14}=x_{12}^2-2, & x_{24}=x_{12}, & x_{34}=x_{12}x_{23}-x_{13}, & 2=x_{12}x_{24}-x_{14},\\
x_{12}=x_{13}^2-2, & 2=x_{13}x_{23}-x_{12}, & x_{23}=x_{13}, &x_{24}=x_{13}x_{34}-x_{14}\\
x_{13}=x_{24}x_{14}-x_{12}, & x_{23}=x_{24}^2-2, & 2=x_{24}x_{34}-x_{23}, & x_{34}=x_{24}.
\end{array}
\]
It follows from solving these equations that $F_2(4_1)$ is parametrized by $x_{13}$ satisfying 
\[
(x_{13}-2)(x_{13}^2+x_{13}-1)=0.  
\]
Thus we obtain the following fundamental variety of $K=4_1$: 
\[
F_2(4_1) \cong \left\{x_{13}=2,(-1\pm\sqrt{5})/2\right\}.
\]
Furthermore, $F_2(4_1)$ is isomorphic to $S_0(4_1)$. Indeed, the relations (GH): 
\[ 
x_{i_1 i_2 i_3}\cdot x_{j_1 j_2 j_3}
=\frac{1}{2}
\left|
\begin{array}{ccc}
x_{i_1 j_1} & x_{i_1 j_2} & x_{i_1 j_3}\\
x_{i_2 j_1} & x_{i_2 j_2} & x_{i_2 j_3}\\
x_{i_3 j_1} & x_{i_3 j_2} & x_{i_3 j_3}
\end{array}
\right| 
\hspace*{1cm} 
\begin{minipage}{4cm}
{\small 
$(1\leq i_1 <i_2 <i_3\leq 4)$\\
$(1\leq j_1 <j_2 <j_3\leq 4)$}
\end{minipage}, 
\]
show that all coordinates $x_{ijk}$ vanish:   
\[
x_{123}=0,\ x_{124}=0,\ x_{134}=0,\ x_{234}=0, 
\] 
For example, using (F2), the relation (GH) for $x_{123}$ can be transformed into $0$ as follows:    
\begin{eqnarray*}
x_{123}^2 &=& \frac{1}{2}\left|
\begin{array}{ccc}
2 & x_{12} & x_{13}\\
x_{21} & 2 & x_{23}\\
x_{31} & x_{32} & 2
\end{array}
\right|=x_{12}x_{13}x_{23}-x_{12}^2-x_{13}^2-x_{23}^2+4\\
&=&(x_{13}^2-2)x_{13}^2-(x_{13}^2-2)^2-x_{13}^2-x_{13}^2+4=0. 
\end{eqnarray*}
The other $x_{ijk}$ can be shown to be $0$, similarly. 
This indicates that every point $(x_{13})$ in $F_2(4_1)$ lifts to the point $(x_{13};0)$ in $S_0(4_1)$, 
and thus $F_2(4_1)$ is isomorphic to $S_0(4_1)$. 
In this process, $S_0(4_1)$ is obtained by first computing $F_2(4_1)$, 
and then verifying the liftability condition given by the relations (GH).
\[
\begin{minipage}{6cm}
\begin{overpic}[width=\hsize]{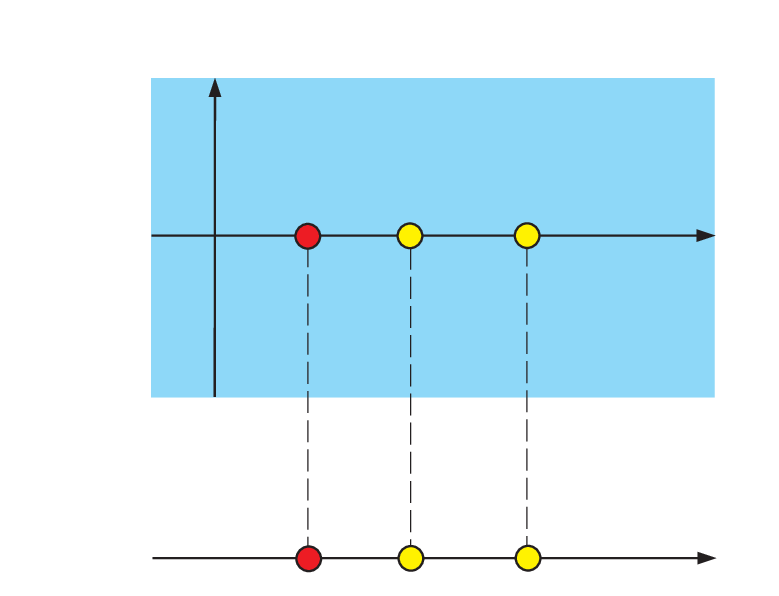}
\put(95,2){\small $\C^{4 \choose 2}$}
\put(95,45){\small $\C^{4 \choose 2}$}
\put(25,70){\small $\C^{4 \choose 3}$}
\put(22,40){\small $0$}
\put(0,3){\hspace*{-0.5cm} \small $F_2(4_1)=$}
\put(0,45){\hspace*{-0.5cm} \small $S_0(4_1)=$}
\end{overpic}
\end{minipage} 
\]

We next consider the case where $K$ is the $5_2$ knot. 
From the diagram below, it follows by a computer calculation 
that $F_2(5_2)$ is parametrized by $x_{14}$, which satisfies the equation 
\[
(x_{14}-2)\left(x_{14}^3+x_{14}^2-2x_{14}-1\right)=0.
\]
By a similar argument, all $x_{ijk}$ are shown to vanish. 
Consequently, every point $(x_{14})$ in $F_2(5_2)$ lifts to the point 
$(x_{14};0)$ in $S_0(5_2)$. Hence we obtain 
\[
S_0(5_2) \cong F_2(5_2)=\left\{x_{14}\in\C \mid 
(x_{14}-2)\left(x_{14}^3+x_{14}^2-2x_{14}-1\right)=0 \right\}. 
\]
\[
\begin{minipage}{3.5cm}\includegraphics[width=\hsize]{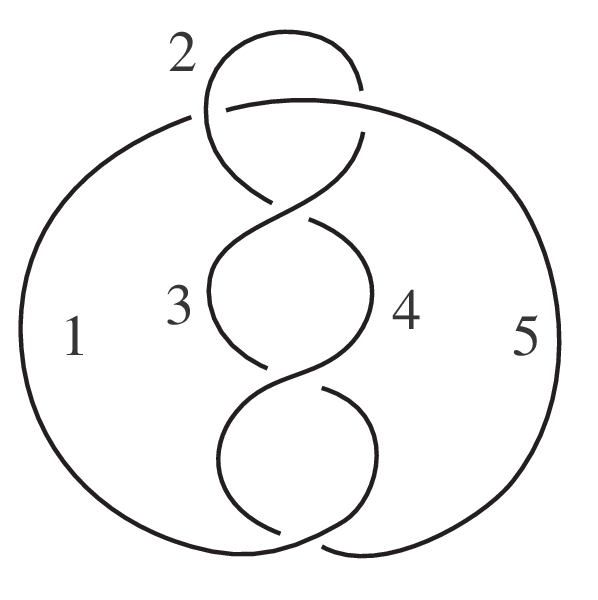}\end{minipage}
\begin{minipage}{1.5cm}\includegraphics[width=\hsize]{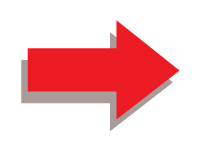}\end{minipage}
\hspace*{1cm}
\begin{minipage}{6cm}
\begin{overpic}[width=\hsize]{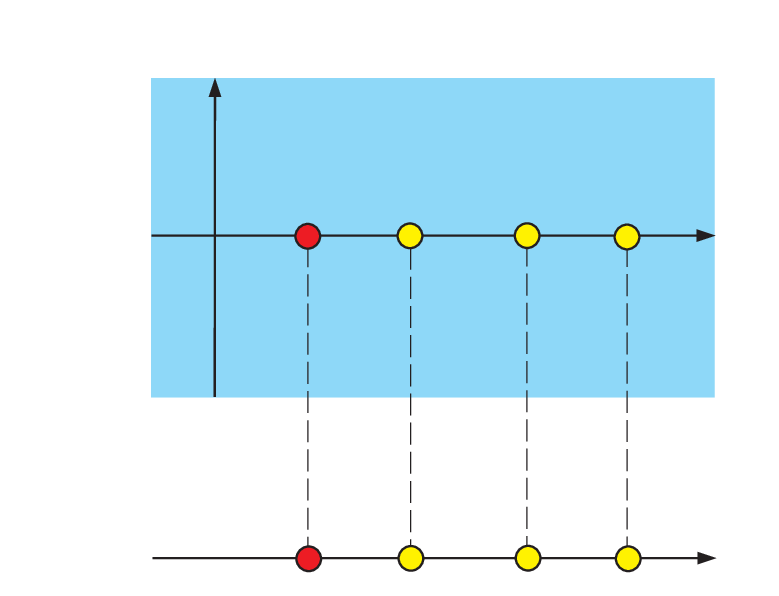}
\put(95,2){\small $\C^{5 \choose 2}$}
\put(95,45){\small $\C^{5 \choose 2}$}
\put(25,70){\small $\C^{5 \choose 3}$}
\put(22,40){\small $0$}
\put(0,3){\hspace*{-0.5cm} \small $F_2(5_2)=$}
\put(0,45){\hspace*{-0.5cm} \small $S_0(5_2)=$}
\end{overpic}
\end{minipage}
\]
As in the previous case, $S_0(5_2)$ is obtained by first determining $F_2(5_2)$
and then verifying the liftability condition given by the relations (GH).

We also analyze the case of the $8_5$ knot by a computer calculation. 
In this case, $F_2(8_5)$ consists of 12 points, whereas $S_0(8_5)$ consists of 13 points.
\[
\begin{minipage}{3.5cm}\includegraphics[width=\hsize]{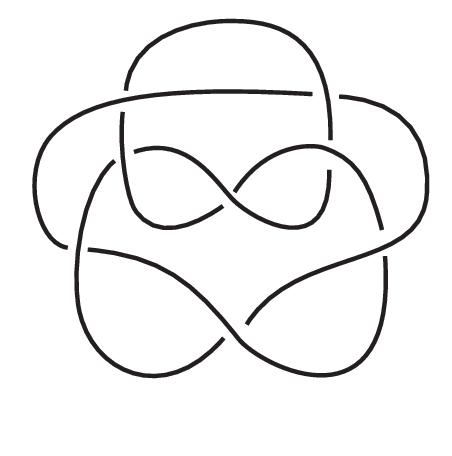}\end{minipage}
\begin{minipage}{1.5cm}\includegraphics[width=\hsize]{arrow.eps}\end{minipage}
\hspace*{1cm}
\begin{minipage}{6cm}
\begin{overpic}[width=\hsize]{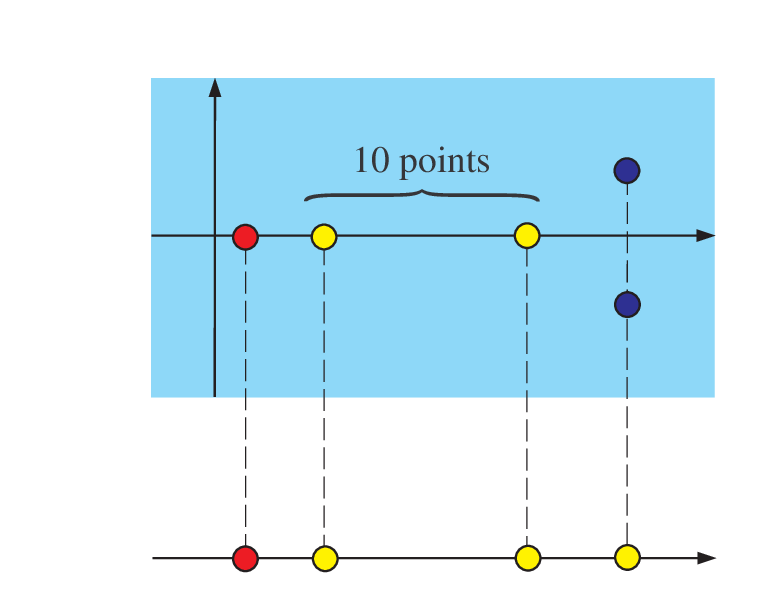}
\put(95,2){\small $\C^{8 \choose 2}$}
\put(95,45){\small $\C^{8 \choose 2}$}
\put(25,70){\small $\C^{8 \choose 3}$}
\put(22,40){\small $0$}
\put(0,3){\hspace*{-0.5cm} \small $F_2(8_5)=$}
\put(0,45){\hspace*{-0.5cm} \small $S_0(8_5)=$}
\end{overpic}
\end{minipage}
\]
As seen in this case, a point in $F_2(K)$ lifts to at most two points
in $S_0(K)$ under the relations (GH).
This endows the trace-free slice $S_0(K)$ with the structure of a 2-fold branched cover 
of the fundamental variety $F_2(K)$. 
The branch set consists of the characters of metabelian representations,
which are equivalent to binary dihedral representations
(see \cite[Theorem 1, Proposition 2]{Nagasato-Yamaguchi}).
Accordingly, the above computation shows that the $8_5$ knot admits
irreducible non-metabelian representations.
The $8_{20}$ knot is also known to admit such representations (see \cite{Nagasato2}),
and Zentner \cite{z} has discovered non-binary dihedral representations
for certain alternating pretzel knots.
Therefore, $F_2(K)$ and $S_0(K)$ are not isomorphic as algebraic sets in general.
The structure of $S_0(K)$ as a 2-fold branched cover of $F_2(K)$ 
plays a crucial role in the study of Ng's conjecture on degree $0$ abelian knot contact homology, 
which will be investigated in Subsection \ref{subsec_ghost}. 


\section{Application of the trace-free slice to abelian knot contact homology}
\label{sec_hc0}
\subsection{Degree 0 abelian knot contact homology}
In \cite{Ng1}, degree $0$ abelian knot contact homology was defined 
via representations of braid groups. 
This definition was later extended to arbitrary knot diagrams, 
not necessarily given as braid closures, in \cite{Ng2}. 
We briefly review this definition for arbitrary knot diagrams, 
following \cite[Section 4.3]{Ng2}. 

Let $K \subset \R^3$ be a knot with an $n$-crossing diagram $D_K$, and let 
\[
G(K)=\langle m_1,\cdots,m_n \mid r_1, \cdots, r_{n-1} \rangle
\] 
be the associated Wirtinger presentation. 
Let $\A_n^{ab}$ denote the polynomial ring over $\Z$ 
generated by indeterminates $a_{ij}$ $(1 \le i < j \le n)$, 
with the convention that $a_{ii}=-2$ for all $i$, where the indices correspond to the meridians 
$m_1,\cdots,m_n$. 
We denote by $\I_{D_K}$ the ideal\footnote{Although the original definition of the ideal 
$\I_{D_K}$ is given in a different form, it is equivalent to the ideal defined in the present paper. 
For details, see \cite{Ng2}; see also \cite[Section 6]{Nagasato3}.} of $\A_n^{ab}$ 
generated by the elements 
$a_{lj}+a_{lk}+a_{li}a_{ij}$, where $1 \leq l \leq n$ and $(i,j,k)$ ranges over all Wirtinger triples in $D_K$: 
\[
\I_{D_K}:=\left\langle
a_{lj}+a_{lk}+a_{li}a_{ij}\ \left|\ 
\mbox{\small $(i,j,k)$: any Wirtinger triple, }  \mbox{$1 \leq l \leq n$}
\right.\right\rangle.
\]
Each generator of $\I_{D_K}$ is symmetric in the indices $j$ and $k$,
for the same reason as the fundamental relations {\rm (F2)}:
\begin{eqnarray*}
&&a_{ij}+a_{ik}-2a_{ij}=a_{ik}-a_{ij},\\
&&a_{lk}+a_{lj}+a_{li}a_{ik}=a_{lj}+a_{ik}+a_{li}a_{ij}-a_{li}(a_{ij}-a_{ik}).
\end{eqnarray*}
We are now ready to define degree $0$ abelian knot contact homology.
\begin{definition}[degree 0 abelian knot contact homology]
Under the above setup, the degree $0$ abelian knot contact homology $HC_0^{ab}(K)$ of a knot $K$ 
is defined by 
\[
\A_n^{ab}/\I_{D_K}=\frac{\Z[a_{12}, \cdots, a_{n n-1}]}{\left\langle
a_{lj}+a_{lk}+a_{li}a_{ij}\ (\mbox{\small $(i,j,k)$: any Wirtinger triple, }  
\mbox{$1 \leq l \leq n$})\right\rangle}.
\] 
\end{definition}
It follows from \cite[Proposition 4.7]{Ng2} that $HC_0^{ab}(K)$ is a knot invariant, 
that is, the ring $\A_n^{ab}/\I_{D_K}$ is independent of the choice of a diagram $D_K$  
up to an isomorphism fixing $\Z$ pointwise. 


\subsection{Fundamental variety and degree 0 abelian knot contact homology}
As seen in Subsection \ref{sec_obs}, for a Wirtinger presentation 
$G(K)=\langle m_1,\cdots,m_n \mid r_1, \cdots, r_{n-1} \rangle$, 
the trace-free slice $S_0(K)$ admits the structure of 
a 2-fold branched cover over the fundamental variety $F_2(K)$:
\begin{eqnarray*}
&&F_2(K)=\left\{(x_{12},\cdots,x_{nn-1}) \in \C^{n \choose 2} 
\left|\ 
\begin{array}{l}
x_{ak}-x_{ij}x_{ai}+x_{aj}=0\\
\mbox{\small $(i,j,k)$: any Wirtinger triple, $1 \leq a \leq n\}$}
\end{array}
\right.\right\}
\end{eqnarray*}
One immediately observes that the defining equations of $F_2(K)$
closely resemble the generators of the ideal $\I_{D_K}$.
This relationship can be explained, using Hilbert's Nullstellensatz, as follows
(cf. \cite{Nagasato3}). 
The coordinate ring $\mathbf{C}[F_2(K)]$ of the fundamental variety $F_2(K)$ 
admits the presentation: 
\[
\mathbf{C}[F_2(K)] \cong \frac{\C[x_{12},\cdots,x_{n n-1}]}
{\sqrt{\langle x_{aj}+x_{ak}-x_{ai}x_{ij},
\mbox{\small $(i,j,k)$: any Wirtinger triple, } \mbox{$1 \leq a \leq n\}$}
\rangle}}. 
\]
It is then straightforward to verify that the map 
$f: HC_0^{ab}(K) \otimes \C \to \mathbf{C}[F_2(K)]$, defined by $f(a_{ij})=-x_{ij}$ and $f(1)=1$,
induces a ring homomorphism. The kernel of $f$ is obviousely the nilradical $\sqrt{0}$.
Consequently, we obtain the following result.
\begin{proposition}[cf. Theorem 7.5 in \cite{Nagasato3}]\label{contact-F2K}
For any knot $K$, there is an isomorphism 
\[
(HC_0^{ab}(K) \otimes \C)/ \sqrt{0} \cong \mathbf{C}[F_2(K)].
\]
\end{proposition}
Since $HC_0^{ab}(K)$ is a knot invariant up to isomorphism fixing $\Z$ pointwise,
$(HC_0^{ab}(K)\otimes\C)/\sqrt{0}$ is also a knot invariant, 
up to isomorphism fixing $\C$ pointwise. 
By Hilbert's Nullstellensatz, this implies that the fundamental variety $F_2(K)$ 
is a knot invariant\footnote{This was also verified directly in \cite{Nagasato3}.}  
up to biregular equivalence. 


\subsection{Ng's conjecture and ghost characters of a knot}
\label{subsec_ghost}
Proposition \ref{contact-F2K} provides insight into the conjecture proposed 
by Ng in \cite{Ng2} concerning the relationship between degree 0 
abelian knot contact homology and the character variety of the 2-fold branched cover 
of $\S^3$ branched along a knot. We briefly recall the relevant constructions below. 

For a knot $K$, let 
\[ 
G(K)=\langle m_1,\cdots,m_n \mid r_1,\cdots,r_{n-1} \rangle
\]
be the knot group generated by meridians\footnote{This is not necessarily 
a Wirtinger presentation of $G(K)$.}. 
Let $p: C_2K \to E_K$ denote the 2-fold cyclic cover 
of the knot exterior $E_K$, chosen so that the image $p(\mu_2)$ of a meridian $\mu_2$ 
of $C_2K$ is homotopic to the square $m_1^2$ of a meridian $m_1$ of $K$. 
The 2-fold branched cover $\Sigma_2K$ of $\S^3$ branched along $K$ is obtained 
from $C_2K$ by trivially filling a solid torus, that is, by attaching the standard meridian 
of the solid torus to $\mu_2$). 
By this construction, we have 
\[
\pi_1(\Sigma_2K) \cong \pi_1(C_2K)/ \l \mu_2 \r,
\] 
where $\l \mu_2 \r$ denotes the normal closure of the group $\langle \mu_2 \rangle$. 
Since the covering map $p$ induces an injection $p_* : \pi_1(C_2K) \to G(K)$, 
it follows that  
\[
\pi_1(\Sigma_2K) \cong \pi_1(C_2K)/ \l \mu_2 \r \cong \mathrm{Im}(p_*)/\l m_1^2 \r. 
\] 
For a group $G$, we denote by $\mathfrak{R}(G)$ the set of $\SL_2(\C)$-representations 
of $G$. 
Accordingly, the set $\mathfrak{R}(\Sigma_2K)=\mathfrak{R}(\pi_1(\Sigma_2K))$ 
can be identified with 
\[
\mathfrak{R}(\Sigma_2K)=\left\{\rho_* \in \mathfrak{R}(\mathrm{Im}(p_*)) 
\mid \rho_*(m_1^2)=E \right\},  
\]
where $E$ denotes the identity matrix. 
Consequently, the set $\mathfrak{X}(\Sigma_2K)=\mathfrak{X}(\pi_1(\Sigma_2K))$ is given by 
\[
\mathfrak{X}(\Sigma_2K)=\{ \chi_{\rho_*} \mid \rho_* \in \mathfrak{R}(\mathrm{Im}(p_*)),\ 
\rho_*(m_1^2)=E \}, 
\]
and the character variety $X(\Sigma_2K)$ is realized as the image 
$t(\mathfrak{X}(\Sigma_2K))$ under the trace map $t$. 

As mentioned in Section \ref{sec_review}, the character variety $X(\Sigma_2K)$ 
is a closed algebraic set, and hence admits the coordinate ring $\mathbf{C}[X(\Sigma_2K)]$. 
Ng's conjecture asserts that the degree $0$ abelian knot contact homology of $K$, after
complexification, recovers this coordinate ring. 

\begin{conjecture}[Conjecture 5.7 in \cite{Ng2}]\label{conj_ng}
Let $G(K)=\langle m_1,\cdots, m_n \mid r_1,\cdots, r_{n-1} \rangle$ 
be a Wirtinger presentation. Then the ring homomorphism
\[
g: HC_0^{ab}(K) \otimes \C \to \mathbf{C}[X(\Sigma_2K)]
\]
given by $g(a_{ij}) = -t_{m_i m_j}$ $(1 \leq i < j \leq n)$, $g(1)=1$, 
is an isomorphism. 
\end{conjecture}
Since the coordinate ring of any closed algebraic set is reduced, 
it is natural to consider the nilradical quotient of the left-hand side. 
In this form, Proposition \ref{contact-F2K} shows that Conjecture \ref{conj_ng} 
is equivalent to the following statement: the ring homomorphism 
\[
h: \mathbf{C}[F_2(K)] \cong (HC_0(K) \otimes \C)/\sqrt{0} \to \mathbf{C}[X(\Sigma_2K)]
\]
defined by $h(x_{ij})=t_{m_im_j}$ $(1 \leq i < j \leq n)$, $h(1)=1$, is an isomorphism. 
By Hilbert's Nullstellensatz, this holds if and only if the following  
formulation is satisfied. 

\begin{conjecture}\label{conj_nag}
Let $G(K)=\langle m_1,\cdots, m_n \mid r_1,\cdots, r_{n-1} \rangle$ 
be a Wirtinger presentation.  Then the pull-back of $h$ 
\[
h^*:X(\Sigma_2K) \to F_2(K) 
\]
given by $v(h^*(z))=h(v)(z)$ for $v \in \mathbf{C}[F_2(K)]$ 
and $z \in X(\Sigma_2K)$, is an isomorphism of algebraic sets (a biregular map).
\end{conjecture}

From this point on, we concentrate on Conjecture \ref{conj_nag} 
and the map $h^*: X(\Sigma_2K) \to F_2(K)$, rather than on Conjecture \ref{conj_ng}. 
A key observation in analyzing the map $h^*$ is that the elements $m_im_j$ 
$(1 \leq i<j \leq n)$ in $G(K)$, which appear in Conjecture \ref{conj_ng}, 
lie in the image $\mathrm{Im}(p_*)$. 
Topologically, this can be seen by constructing a loop $\gamma$ in $C_2K$ 
whose homotopy class $[\gamma]$ satisfies $p_*([\gamma])=m_im_j$. 
For instance, let $S$ be a regular Seifert surface for $K$, which is isotopic to a disk with  
braided bands as shown in \cite[Lemma 1.8]{Lin2}, and fix a base point $b$ on $S$. 
Then, up to homotopy, the loop representing $m_im_j$ may be chosen so as to intersect 
the surface $S \cap E_K$ in exactly two points (see Figure \ref{fig_mimj}). 

\begin{figure}[htbp]
\[
\begin{minipage}{12cm}
\begin{overpic}[width=\hsize]{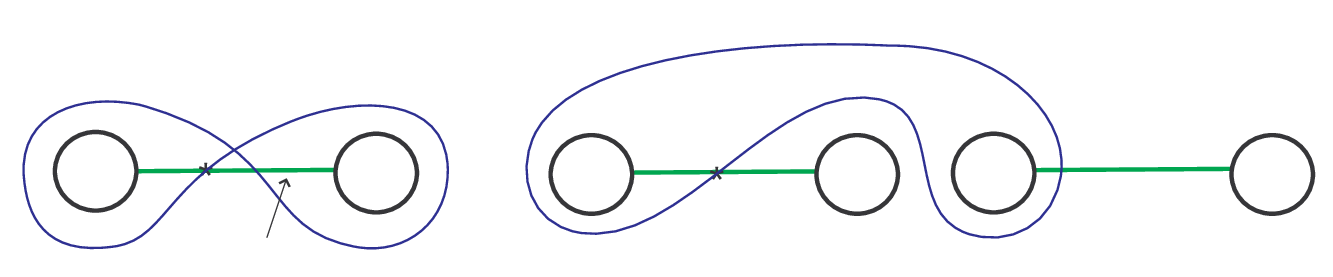}
\put(15,4){\small $b$}
\put(53,9){\small $b$}
\put(5,7){\small $N_K$}
\put(26,7){\small $N_K$}
\put(42,7){\small $N_K$}
\put(62,7){\small $N_K$}
\put(72,7){\small $N_K$}
\put(93,7){\small $N_K$}
\put(13,0){\small $S \cap E_K$}
\put(51,3){\small $S \cap E_K$}
\put(81,3){\small $S \cap E_K$}
\put(22,16){$m_im_j$}
\put(54,19){$m_im_j$}
\end{overpic}
\end{minipage}
\]
\caption{Intersetion patterns of the loop representing $m_im_j$ 
with the surface $S \cap E_K$. $N_K$ denotes a tubular neighborhood of $K$.}
\label{fig_mimj}
\end{figure}

In this setting, by construction, the element $m_im_j$ admits a unique lift to $C_2K$ 
up to homotopy. We denote this lift by $\gamma$. 
This construction\footnote{A similar construction was introduced earlier in \cite{Ng2} 
in the definition of the map $g$.} satisfies $p_*([\gamma])=m_im_j$. 
In particular, this shows that $m_im_j$ $(1 \leq i < j \leq n)$ lies in $\mathrm{Im}(p_*)$. 
The above topological observation is supported by the following algebraic description 
due to Fox. 

\begin{theorem}[\cite{Fox}, cf.\cite{Kinoshita}]\label{thm_Fox}
For a knot $K$, let $G(K)=\langle m_1,\cdots,m_n \mid r_1,\cdots,r_{n-1} \rangle$ 
be the knot group generated by $n$ meridians $m_1,\dots,m_n$. 
Then we have 
\[
\pi_1(\Sigma_2K) \cong 
\langle m_1m_i\ (2 \leq i \leq n) 
\mid w(r_j),w(m_1r_jm_1^{-1})\ (1 \leq j \leq n-1), 
m_i^2\ (1 \leq i \leq n)\rangle,
\] 
where $w(r_j)$ (resp. $w(m_1r_jm_1^{-1})$) denotes the word obtained by expressing $r_j$ 
(resp. $m_1r_jm_1^{-1}$) in terms of the generators $m_1m_2,\cdots,m_1m_n$. 
\end{theorem}

Theorem \ref{thm_Fox} can be obtained from the following presentation 
of $\mathrm{Im}(p_*)$:  
\[
\mathrm{Im}(p_*) 
\cong 
\langle m_1m_i, m_im_1^{-1} (1 \leq i \leq n) 
\mid w(r_j),w(m_1r_jm_1^{-1})\ (1 \leq j \leq n-1)\rangle. 
\]
This arises from the coset decomposition 
$G(K)=\mathrm{Im}(p_*) \cup \mathrm{Im}(p_*)m_1$ 
via the Schreier system $\{1,m_1\}$, together with the injection $p_*: \pi_1(C_2K) \to G(K)$. 
Taking the quotient\footnote{Any meridian $m_i$ may be used in place of $m_1$, 
since all meridians are conjugate.} by the normal closure $\l m_1^2 \r$ and 
applying Tietze transformations then yield the presentation in Theorem \ref{thm_Fox}. 
For further details, see \cite{Fox,Kinoshita}. 

Using the presentation in Theorem \ref{thm_Fox} and the trace map $t$, 
the character variety $X(\Sigma_2K)$ is parametrized as follows. 
For a representation $\rho_*:\pi_1(\Sigma_2K)\to \SL_2(\C)$, set
\begin{eqnarray*}
y_{a}(\chi_{\rho_*}) &:=&t_{m_1m_a}(\chi_{\rho_*}),\\
y_{ab}(\chi_{\rho_*}) &:=&t_{(m_1m_a)(m_1m_b)}(\chi_{\rho_*}),\\
y_{abc}(\chi_{\rho_*}) &:=&t_{(m_1m_a)(m_1m_b)(m_1m_c)}(\chi_{\rho_*}). 
\end{eqnarray*}
Then, as shown in Section \ref{sec_review}, the character variety $X(\Sigma_2K)$ 
can be realized via the trace map $t$ as 
\[
X(\Sigma_2K)=\left\{\left(y_{a}(\chi_{\rho_*}); y_{bc}(\chi_{\rho_*}); y_{def}(\chi_{\rho_*})\right) 
\in \C^{n-1+{n-1 \choose 2}+{n-1 \choose 3}}\ 
\left|\ 
\begin{array}{l}
\chi_{\rho_*} \in \mathfrak{X}(\Sigma_2K)\\
2 \leq a \leq n\\
2 \leq b < c \leq n \\
2 \leq d < e < f \leq n
\end{array}
\right.\right\}. 
\]
For $1 \leq a < b \leq n$, we define the function $z_{ab}$ on $\mathfrak{X}(\Sigma_K)$ by 
\[
z_{ab}(\chi_{\rho_*}):=t_{m_am_b}(\chi_{\rho_*})=t_{(m_1m_a)^{-1}(m_1m_b)}(\chi_{\rho_*})
=y_a(\chi_{\rho_*})y_b(\chi_{\rho_*})-y_{ab}(\chi_{\rho_*}). 
\] 
This yields an equivalent parametrization of $X(\Sigma_2K)$: 
\[
X(\Sigma_2K) \cong 
\left\{\left(z_{ab}(\chi_{\rho_*}); y_{def}(\chi_{\rho_*}) \right) 
\in \C^{{n \choose 2}+{n-1 \choose 3}}\ 
\left|\ 
\begin{array}{ll}
\chi_{\rho_*} \in \mathfrak{X}(\Sigma_2K), & 1 \leq a < b \leq n\\
& 2 \leq d < e < f \leq n
\end{array}
\right.\right\}. 
\]

For the remainder of this section, we assume that
\[
G(K)=\langle m_1,\cdots,m_n \mid r_1,\cdots,r_{n-1} \rangle
\]
is a Wirtinger presentation. Then, with respect to the coordinates $(z_{ab};y_{def})$, 
the map $h^*:X(\Sigma_2K) \to F_2(K)$ is explicitly given by 
\[
h^*((z_{ab};y_{def}))=(z_{ab}). 
\] 
Indeed, for any $1 \leq i < j \leq n$ 
and $\chi_{\rho_*}=(z_{ab};y_{def}) \in X(\Sigma_2K)$, we have
\[
x_{ij}(h^*((z_{ab};y_{def})))=h(x_{ij})(\chi_{\rho_*})=t_{m_im_j}(\chi_{\rho_*})=z_{ij}.
\]
We note that the polynomial map $h^*$ is well-defined. 
To see this, for any Wirtinger triple $(i,j,k)$ and any $1 \leq a \leq n$, 
the relations 
\[
m_am_k=(m_a m_i)(m_j m_i),\ m_l^2=1\ (1 \leq l \leq n)
\] 
in $\pi_1(\Sigma_2K) \cong \mathrm{Im}(p_*)/\l m_1^2 \r$ imply that, 
for any character $\chi_{\rho_*} \in X(\Sigma_2K)$, 
\[
t_{m_am_k}(\chi_{\rho_*})
=t_{m_am_i}(\chi_{\rho_*})t_{m_im_j}(\chi_{\rho_*})-t_{m_am_j}(\chi_{\rho_*}). 
\]
Consequently, the point $h^*(\chi_{\rho_*})=(t_{m_im_j}(\chi_{\rho_*})) \in \C^{n \choose 2}$ 
satisfies the relations (F2) for $G(K)$. 
Hence the image of any point of $X(\Sigma_2K)$ under $h^*$ lies in $F_2(K)$. 

We now turn to Conjecture \ref{conj_nag}. We analyze the conjecture via the map 
\[
\widehat{\Phi}: \mathfrak{S}_0(K) \to \mathfrak{X}(\Sigma_2K),
\] 
constructed in \cite{Nagasato-Yamaguchi}. 
For $\chi_{\rho} \in \mathfrak{S}_0(K)$ and $g \in \pi_1(\Sigma_2K)$, 
the map $\widehat{\Phi}$ is defined by 
\begin{eqnarray*}
\widehat{\Phi}(\chi_{\rho})(g)=(\sqrt{-1})^{\alpha(p_*(g))}\chi_{\rho}(p_*(g)), 
\end{eqnarray*}
where $\alpha: G(K) \to H_1(E_K)=\langle m_1 \rangle \cong \Z$ denotes the abelianization. 
The map $\widehat{\Phi}$ is one-to-one for the characters of metabelian representations 
and two-to-one for the others (see \cite[Theorem 1]{Nagasato-Yamaguchi}). 
Moreover, the map $\widehat{\Phi}$ is surjective for 
any 2-bridge knots \cite[Proposition 11]{Nagasato-Yamaguchi} 
and for pretzel knots \cite[Proposition 14]{Nagasato-Yamaguchi}, which are 3-bridge knots. 

With respect to the above parametrization $(z_{ab};y_{def})$ of $X(\Sigma_2K)$, 
the map $\widehat{\Phi}$ can be described explicitly as a polynomial map. 
For a trace-free character $\chi_{\rho}=(x_{ij};x_{ijk})\in S_0(K)$, we have
\begin{eqnarray*}
\widehat{\Phi}((x_{ij};x_{ijk}))
&=&\left(t_{m_am_b}(\widehat{\Phi}(\chi_{\rho}));
t_{(m_1m_d)(m_1m_e)(m_1m_f)}(\widehat{\Phi}(\chi_{\rho}))
\right)
\label{poly_phi}\\
&=&\left(x_{ab}; x_{1d}x_{1e}x_{1f}-\frac{1}{2}(x_{1d}x_{ef}+x_{1e}x_{df}+x_{1f}x_{de})
\right),\nonumber
\end{eqnarray*}
where we have used the $\SL_2(\C)$-trace identity (equivalently, the skein relations)  
together with the trace-free condition. Indeed, for any trace-free character 
$\chi_{\rho} \in S_0(K)$, we obtain  
\begin{eqnarray*}
t_{m_a m_b}(\widehat{\Phi}(\chi_{\rho}))&=&t_{m_a m_b}(-\chi_{\rho})
=-t_{m_am_b}(\chi_{\rho})\\
t_{(m_1m_d)(m_1m_e)(m_1m_f)}(\widehat{\Phi}(\chi_{\rho}))
&=&t_{(m_1m_d)(m_1m_e)(m_1m_f)}(-\chi_{\rho})=-t_{(m_1m_d)(m_1m_e)(m_1m_f)}(\chi_{\rho})\\
&=&-t_{(m_1m_d)}(\chi_{\rho})t_{(m_1m_e)}(\chi_{\rho})t_{(m_1m_f)}(\chi_{\rho})\\
&&-\frac{1}{2}t_{(m_1m_d)}(\chi_{\rho})t_{(m_em_f)}(\chi_{\rho})\\
&&-\frac{1}{2}t_{(m_1m_e)}(\chi_{\rho})t_{(m_dm_f)}(\chi_{\rho})\\
&&-\frac{1}{2}t_{(m_1m_f)}(\chi_{\rho})t_{(m_dm_e)}(\chi_{\rho}). 
\end{eqnarray*}
Consequently, the image of a point $(x_{ij};x_{ijk}) \in S_0(K)$ 
under the map $\widehat{\Phi}$ is completely determined by the coordinates $(x_{ij})$. 

Let $q: S_0(K) \to F_2(K)$ be the projection that forgets the coordinates $(x_{ijk})$. 
This induces a polynomial map 
\[
r: \mathrm{Im}(q) \to X(\Sigma_2K), 
\]
defined by 
\[
r(x_{ij})=\widehat{\Phi}(q^{-1}((x_{ij})))
=\left(x_{ab}; x_{1d}x_{1e}x_{1f}-\frac{1}{2}(x_{1d}x_{ef}+x_{1e}x_{df}+x_{1f}x_{de})
\right). 
\]
This explicit form shows that $r$ is injective. 
Hence, $\mathrm{Im}(q)$ and $\mathrm{Im}(\widehat{\Phi})$ are in one-to-one 
correspondence\footnote{It is not difficult to show that these are closed algebraic sets 
and hence isomorphic as algebraic sets.} via the maps $r$ and $h^*$. 
This yields the following result. 

\begin{proposition}\label{prop_conj}
If both maps $q$ and $\widehat{\Phi}$ are surjective, then 
the compositions $h^* \circ r$ and $r \circ h^*$ are well-defined, and satisfy 
\[
h^* \circ r=id_{F_2(K)},\ r \circ h^*=id_{X(\Sigma_2K)}.
\] 
In particular, $F_2(K)$ and $X(\Sigma_2K)$ are isomorphic, 
and Conjecture $\ref{conj_nag}$ holds in this case. 
\end{proposition}

As mentioned above, it has been shown that $\widehat{\Phi}$ is surjective for 2-bridge knots 
and for pretzel knots, which is a family of 3-bridge knots. 
More generally,  in Theorem \ref{thm_Ng_conj} (1) we will show 
that the map $\widehat{\Phi}$ is surjective for all 3-bridge knots. 

On the other hand, as observed in Section \ref{sec_obs}, 
the projection $q: S_0(K) \to F_2(K)$ is surjective for several knots. 
It is natural to ask whether this property holds for all knots. 
If there exists a point in $F_2(K)$ whose preimage under $q$ is empty, 
namely, a point that does not lift to $S_0(K)$, 
then we call such a point a ghost character of the knot $K$.

\begin{definition}[Ghost characters of a knot]\label{def_ghost}
A point $(x_{ij})$ in $F_2(K)$ that does not satisfy one of the relations (GH), 
is called a ghost character of the knot $K$. 
\end{definition}

\begin{figure}[htbp]
\[
\begin{minipage}{11.5cm}\begin{overpic}[width=\hsize]{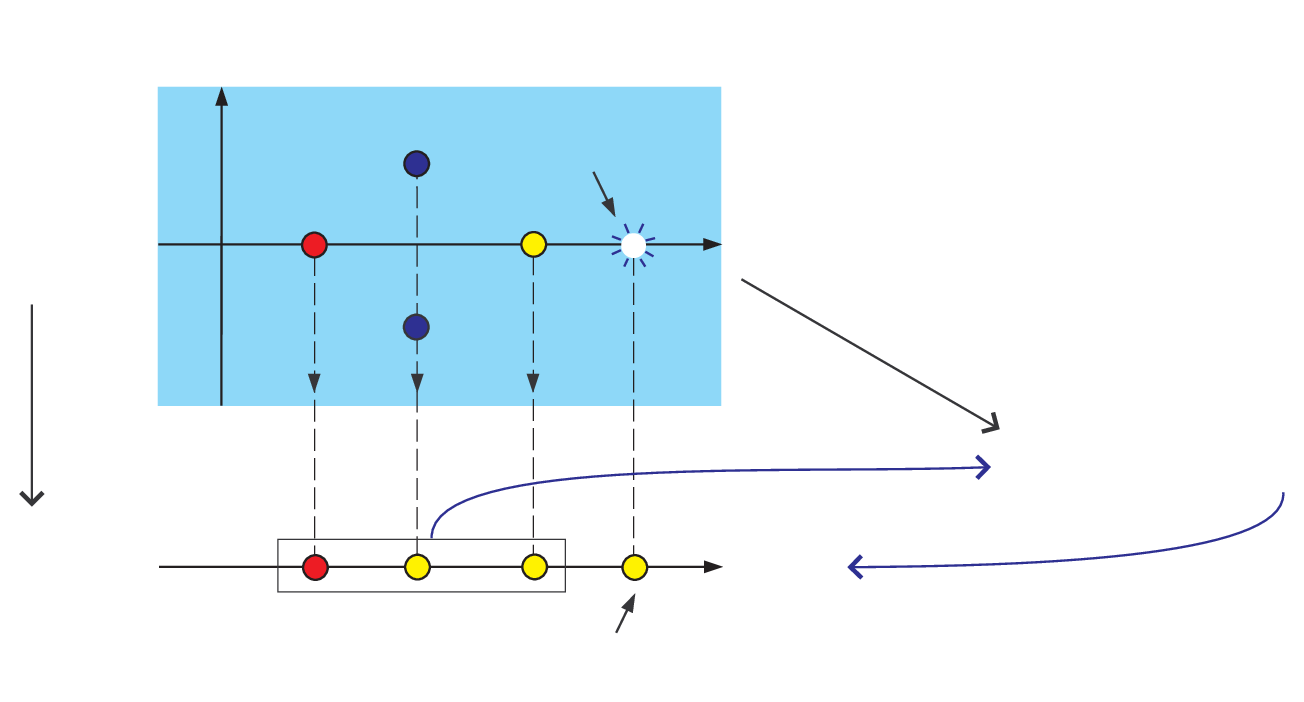}
\put(58,9){\small $\C^{n \choose 2}$}
\put(58,33){\small $\C^{n \choose 2}$}
\put(15,49){\small $\C^{n \choose 3}$}
\put(14,32){\small $0$}
\put(0,10){\hspace*{-0.5cm} \small $F_2(K)=$}
\put(0,34){\hspace*{-0.5cm} \small $S_0(K)=$}
\put(67,28){$\widehat{\Phi}$ {\scriptsize (2-fold branched)}}
\put(4,22){$q$}
\put(24,5){\small $\mathrm{Im}(q)$}
\put(35,2){\small a ghost character}
\put(78,18){\small $\fbox{$\mathrm{Im}(\widehat{\Phi})$}\subset X(\Sigma_2K)$}
\put(47,19){$r$ {\scriptsize (one-to-one)}}
\put(79,12){\small $h^*$}
\put(38,42){\small does not lift!}
\end{overpic}
\end{minipage}
\]
\caption{Schematic view of $X(\Sigma_2K)$, $S_0(K)$ and $F_2(K)$ with a ghost character.}
\label{fig_landscape_S0K}
\end{figure}

By definition, if a knot $K$ admits no ghost characters, then the map
$q : S_0(K) \to F_2(K)$ is surjective.
In fact, we prove the following. 

\begin{theorem}\label{thm_noghost}
Any knot $K$ with bridge index less than $4$ admits no ghost characters.
\end{theorem}

\begin{proof}
We begin by outlining the strategy of the proof. 
Let $K$ be a $3$-bridge knot given by an $n$-crossing diagram $D_K$ in $3$-bridge position,
and let $G(K)=\langle m_1,\dots,m_n \mid r_1,\dots,r_{n-1}\rangle$ 
be again the associated Wirtinger presentation.
The trace-free slice $S_0(K)$ is parametrized via the map $\tilde{t}$ by
\[
\left(-t_{m_im_j}(\chi_\rho);\,-t_{m_im_jm_k}(\chi_\rho)\right)
\in \C^{{n \choose 2}+{n \choose 3}},
\]
where $\rho: G(K) \to \SL_2(\C)$ is a trace-free representation.
Since $K$ is a $3$-bridge knot, the presentation of $G(K)$ can be reduced,
via Tietze transformations, to
\[
G(K)=\langle m_1,m_2,m_3 \mid r_1,r_2\rangle,
\]
reflecting the $3$-bridge structure of $K$.
Accordingly, the above parametrization of $S_0(K)$ reduces to
\[
\bigl(
-t_{m_1m_2}(\chi_\rho),
-t_{m_1m_3}(\chi_\rho),
-t_{m_2m_3}(\chi_\rho);
-t_{m_1m_2m_3}(\chi_\rho)
\bigr)
\in \C^{4}. 
\]
Then, it follows that the fundamental variety $F_2(K)$ admits the description
\[
F_2(K) \cong
\{(x_{12},x_{13},x_{23})\in \C^3 \mid \text{\rm (A), (B), (C)}\},
\]
where (A), (B), and (C) are certain polynomial relations in $x_{12},x_{13},x_{23}$.
In this setting, the relations (GH) reduce to the single relation
\[
x_{123}^2
=\frac12
\begin{vmatrix}
x_{11} & x_{12} & x_{13}\\
x_{21} & x_{22} & x_{23}\\
x_{31} & x_{32} & x_{33}
\end{vmatrix}.
\]
Consequently, every point of $F_2(K) \subset \C^3$ admits a lift to $S_0(K) \subset \C^4$,
and hence the knot $K$ admits no ghost characters.

In the following, we explicitly demonstrate how $S_0(K)$, originally embedded 
in $\C^{{n \choose 2}+{n \choose 3}}$, can be reduced to the above description in $\C^4$. 
The proof proceeds in 3 steps. 
First, we explain how the Kauffman bracket skein theory is used to carry out 
elimination of the defining polynomials of $F_2(K)$ in a diagrammatic way.
Second, we present an explicit elimination process for the fundamental variety 
$F_2(K)$ of a knot $K$ in 3-bridge position, thereby recovering 
the defining equations (A), (B), and (C) introduced above.
Finally, using this reduced parametrization of $F_2(K)$, we show that every point in $F_2(K)$ 
lifts to $S_0(K)$, which completes the proof. 

\fbox{\bf Step 1}
The basic idea for reducing the parameters $x_{ij}$ of $F_2(K) \subset \C^{n \choose 2}$ 
to $x_{12}$, $x_{13}$, and $x_{23}$ is to repeatedly apply the relations (F2) as follows. 
We first eliminate $x_{an}$ $(1 \leq a \leq n-1)$ from (F2) using
\[
x_{an}=x_{p_1q_1}x_{ap_1}-x_{aq_1},
\]
where $(p_1,q_1,n)$ is a Wirtinger triple of $D_K$. 
Next, we eliminate $x_{an-1}$ $(1 \leq a \leq n-2)$ 
from (F2) and the equations obtained in the previous step, using
\[
x_{an-1}=x_{p_2q_2}x_{ap_2}-x_{aq_2},
\]
where $(p_2,q_2,n-1)$ is another Wirtinger triple.
Iterating this procedure down to $x_{a4}$ for $1 \leq a \leq 3$,
we eventually express every $x_{ab}$ $(4 \leq a\leq n$ or $4 \leq b \leq n)$ 
as a polynomial in $x_{12},x_{13},x_{23}$.
This recursive elimination substantially simplifies the parametrization of $S_0(K)$.

In fact, the above elimination process can be understood naturally 
as the following diagrammatic operation, which reflects the diagram $D_K$ in bridge position. 
For $x_{ij}$ with $1 \leq i \leq j \leq n$, consider a loop $s_{ij}$ in the knot exterior $E_K$, 
freely homotopic to $m_im_j$. We decompose $s_{ak}$ into two arcs $c_a$ and $c_k$, 
corresponding to the meridians $m_a$ and $m_k$, respectively. 
Then, as illustrated in Figure \ref{fig_Tietze}, for a Wirtinger triple $(i,j,k)$ of $D_K$,
we slide $c_k$ along the $k$th arc of $D_K$, keeping $c_a$
and the endpoints of $c_k$ fixed. 

\begin{figure}[htbp]
\[
\hspace*{-2cm}
D_K=
\begin{minipage}{6cm}
\begin{overpic}[width=\hsize]{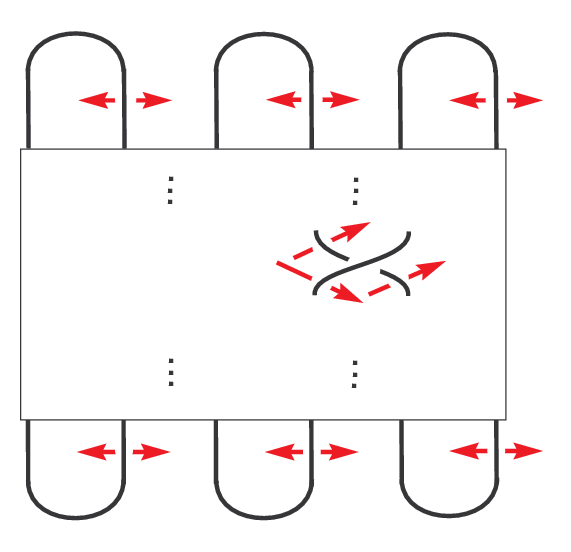}
\put(24,8){$m_1$}
\put(56,8){$m_2$}
\put(89,8){$m_3$}
\put(45,42){$m_i$} 
\put(70,40){$m_j$}
\put(50,60){$m_k$}
\put(22,89){$m_{n-2}$} 
\put(54,89){$m_{n-1}$}
\put(87,89){$m_n$}
\end{overpic}
\end{minipage}
\hspace*{-1.2cm}
\begin{minipage}{2cm}\includegraphics[width=\hsize]{arrow.eps}\end{minipage}
\begin{minipage}{4cm}
\begin{overpic}[width=\hsize]{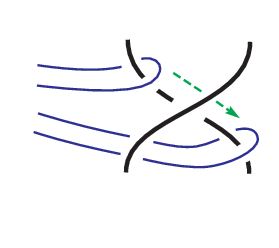}
\put(0,9){$s_{ak}=s_{ij}s_{ai}-s_{aj}$ $(1 \leq a \leq n)$}
\put(10,75){\hspace*{-0.3cm} \small $c_k$ in $s_{ak}$} 
\put(88,49){\small sliding}
\end{overpic}
\end{minipage}
\]
\caption{Sliding the subarc $c_k$ of $s_{ak}$ and the resulting relation. 
Meridians are oriented so that they are conjugate.}
\label{fig_Tietze}
\end{figure}

When the slid arc passes under the $i$th arc of $D_K$, 
the resulting winding part is resolved by applying the trace-free skein relation (\ref{ff-formula}),  
yielding $s_{ak}=s_{ij}s_{ai}-s_{aj}$: 
\[
\begin{minipage}{3.8cm}
\begin{overpic}[width=\hsize]{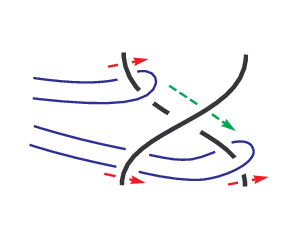}
\put(2,63){\small $c_a \cup c_k$} 
\put(85,42){\small sliding} 
\put(38,8){\small $i$}
\put(80,8){\small $j$}
\put(38,70){\small $k$}
\end{overpic}
\end{minipage}
\hspace*{1cm}
=
\begin{minipage}{3.5cm}
\begin{overpic}[width=\hsize]{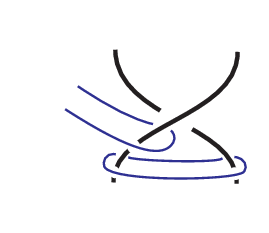}
\put(2,58){\small $c_a \cup c_i$} 
\put(58,8){$s_{ij}$}
\put(40,8){\small $i$}
\put(88,8){\small $j$}
\put(38,73){\small $k$}
\end{overpic}
\end{minipage}
-\begin{minipage}{3.5cm}
\begin{overpic}[width=\hsize]{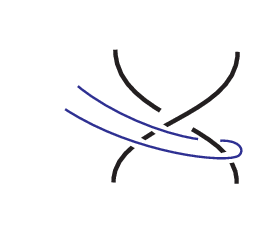}
\put(2,60){\small $c_a \cup c_j$} 
\put(40,8){\small $i$}
\put(88,8){\small $j$}
\put(38,73){\small $k$}
\end{overpic}
\end{minipage}.
\]
After substituting $s_{ij}=x_{ij}$, this becomes the fundamental relation 
(F2): $x_{ak}=x_{ij}x_{ai}-x_{aj}$ for a Wirtinger triple $(i,j,k)$. 

With this setup, the polynomial expression of the parameter $x_{ak}$ in $x_{12}$, $x_{13}$, 
and $x_{23}$ can equivalently be obtained as follows. 
First, place the corresponding loop $s_{ak}$ in $E_K$ 
and slide the subarc $c_k$ of $s_{ak}$ to the bottom of the diagram $D_K$, 
while deferring the resolution of any winding parts encountered during the slide, 
and labeling these winding parts consecutively as they arise. 
Then they are resolved at the end in this order by the trace-free skein relation (\ref{ff-formula}). 
If a labeled part is no longer a winding part at the bottom of $D_K$, 
the label is discarded and no resolution is applied. 
Throughout this paper, we use the phrase ``resolve at the end'' 
to mean this convention. 

\fbox{\bf Step 2}
We summarize the above diagrammatic elimination procedure case by case 
for each parameter $x_{ij}$ ($1 \le i \le j \le n$). 
\begin{enumerate}
\item For $x_{ij}$ with $1 \le i < j \le n-3$, we may place the corresponding loop $s_{ij}$ 
below the top strands of the diagram $D_K$. 
We slide $s_{ij}$ downward to the bottom of $D_K$ and resolve the winding parts 
at the end. 
After substituting $s_{ij}=x_{ij}$, the resulting expression eliminates $x_{ij}$. 
The case $i=j$ yields the trivial relation $x_{ii}=2$ and is therefore omitted. 

\item For $x_{ij}$ with $1 \le i < n-2 \le j \le n$, the corresponding loop $s_{ij}$ 
admits two possible downward sliding paths. 
This is because the subarc $c_j$ lies at the top of the diagram $D_K$, as illustrated below: 
\[
\begin{minipage}{7cm}
\begin{overpic}[width=\hsize]{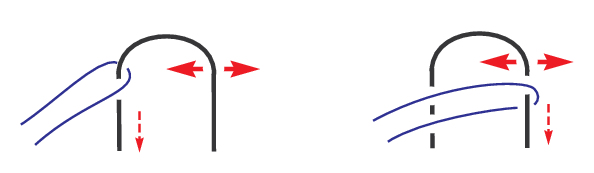}
\put(41,22){\small $m_i$}
\put(50,10){\small or}
\put(93,22){\small $m_i$}
\end{overpic}
\end{minipage}.
\]
Let $R_j(s_{ij})$ (resp.\ $L_j(s_{ij})$) denote the polynomial in $s_{12},s_{13},s_{23}$ 
obtained by choosing the right (resp. left) sliding path for $c_j$.
It follows that $R_j(x_{ij})=L_j(x_{ij})$ by construction. 
This equality eliminates $x_{ij}$ and provides 
one of the defining equations for $F_2(K) \subset \C^3$. 

\item For $x_{ij}$ with $n-2 \le i < j \le n$, the corresponding loop $s_{ij}$ admits 
four possible sliding paths, as each of the subarcs $c_i$ and $c_j$ can be slid 
either to the left or to the right. 
We denote the resulting polynomials in $s_{12}, s_{13}, s_{23}$ by 
$R_iR_j(s_{ij})$, $R_iL_j(s_{ij})$, $L_iR_j(s_{ij})$, $L_iL_j(s_{ij})$, 
according to the choice of left or right sliding for $c_i$ and $c_j$.
By construction, they yield the relations
\[
R_iR_j(x_{ij})=R_iL_j(x_{ij})=L_iR_j(x_{ij})=L_iL_j(x_{ij}). 
\]
These eliminate $x_{ij}$ and provide additional defining equations for $F_2(K) \subset \C^3$. 

\item For $x_{ij}$ with $n-2 \le i = j \le n$, the corresponding loop $s_{ij}$ admits again 
four possible sliding paths, as in Process (3). 
However, it suffices to consider only the case $R_iL_j(x_{ij})$, 
since both $R_iR_j(x_{ij})$ and $L_iL_j(x_{ij})$ yield the trivial relation $x_{ii}=2$, 
and the equality $L_iR_j(x_{ij})=R_iL_j(x_{ij})$ holds, by construction. 
Thus we obtain the single equation $R_iL_j(x_{ij})=2$, yielding 
the remaining defining equations for $F_2(K) \subset \C^3$. 
\end{enumerate}

By the above arguments, the projection $i: F_2(K) \to \C^3$ defined by  
\[
(x_{12},\ldots,x_{n-1,n}) \mapsto (x_{12},x_{13},x_{23}) 
\]
induces a biregular map onto its image.
Under this projection, the equations obtained in Processes (2), (3), and (4)
give all defining relations of $\mathrm{Im}(i)$.
Consequently, we obtain the isomorphism
\[
F_2(K) \cong \{(x_{12},x_{13},x_{23}) \in \C^3 \mid \text{\rm (A), (B), (C)}\}, 
\]
where the defining equations (A), (B), and (C) are explicitly given by
\[
\begin{array}{lll}
\mbox{\rm (A): for $1 \leq i < n-2 \leq j \leq n$,} 
& R_j(x_{ij}) = L_j(x_{ij}), & \\
\mbox{\rm (B): for $n-2 \leq i < j \leq n$,} 
& R_iR_j(x_{ij})=R_iL_j(x_{ij})=L_iR_j(x_{ij})=L_iL_j(x_{ij}),\\
\mbox{\rm (C): for $n-2 \leq i = j \leq n$,} 
& R_iL_j(x_{ij})=2. 
\end{array}
\]
We remark that this description of $F_2(K)$ extends naturally to a knot $K$
in $m$-bridge position. 

\fbox{\bf Step 3} To reduce the parameters of the trace-free slice 
$S_0(K) \subset \C^{\binom{n}{2}+\binom{n}{3}}$,
we reduce the relations (GH) by the elimination process for $x_{ij}$. 
For instance, each parameter $x_{ap}$ $(1 \leq a < b \leq n)$ can be written in the form 
\[
x_{ab} = \sum_{i=1}^{3} f_{bi}\, x_{ai},
\]
where each $f_{bi}$ is a polynomial in $x_{12}, x_{13}, x_{23}$. 
This expression is obtained by sliding the subarc $c_b$ of the corresponding loop $s_{ab}$ downward, 
while keeping the subarc $c_a$ and the endpoints of $c_b$ fixed, 
and resolving the winding parts at the end. 
Then, the relation (GH) for $x_{123}x_{pqr}$ $(1 \leq p<q<r \leq n)$ 
can be rewritten in the following form:
\begin{eqnarray*}
x_{123}x_{pqr}&=&\frac{1}{2}
\left|\begin{array}{ccc}
x_{1p} & x_{1q} & x_{1r}\\
x_{2p} & x_{2q} & x_{2r}\\
x_{3p} & x_{3q} & x_{3r}
\end{array}\right|
=\frac{1}{2}
\left|\begin{array}{ccc}
\Sigma_{i=1}^3 f_{pi}x_{1i} & \Sigma_{j=1}^3 f_{qj}x_{1j} & \Sigma_{k=1}^3 f_{rk}x_{1k}\\
\Sigma_{i=1}^3 f_{pi}x_{2i} & \Sigma_{j=1}^3 f_{qj}x_{2j} & \Sigma_{k=1}^3 f_{rk}x_{2k}\\
\Sigma_{i=1}^3 f_{pi}x_{3i} & \Sigma_{j=1}^3 f_{qj}x_{3j} & \Sigma_{k=1}^3 f_{rk}x_{3k}
\end{array}\right|\\
&=&\frac{1}{2}
{\small 
\left|\begin{array}{ccc}
x_{11} & x_{12} & x_{13}\\
x_{21} & x_{22} & x_{23}\\
x_{31} & x_{32} & x_{33}
\end{array}\right|
\left|\begin{array}{ccc}
f_{p1} & f_{q1} & f_{r1}\\
f_{p2} & f_{q2} & f_{r2}\\
f_{p3} & f_{q3} & f_{r3}
\end{array}\right|
=
x_{123}^2
\left|\begin{array}{ccc}
f_{p1} & f_{q1} & f_{r1}\\
f_{p2} & f_{q2} & f_{r2}\\
f_{p3} & f_{q3} & f_{r3}
\end{array}\right|.}
\end{eqnarray*}
Thus $x_{123}$ and $x_{pqr}$ for $(p,q,r) \neq (1,2,3)$ satisfy  
\[
x_{123}=
\pm\frac{1}{\sqrt{2}}
\left|
\begin{array}{ccc}
x_{11} & x_{12} & x_{13}\\
x_{21} & x_{22} & x_{23}\\
x_{31} & x_{32} & x_{33}
\end{array}
\right|^{\frac{1}{2}},\ 
x_{pqr}=
x_{123}
{\small \left|\begin{array}{ccc}
f_{p1} & f_{q1} & f_{r1}\\
f_{p2} & f_{q2} & f_{r2}\\
f_{p3} & f_{q3} & f_{r3}
\end{array}\right|}. 
\]
Hence, $S_0(K)$ is isomorphic to the subset of $i(F_2(K)) \times \C \subset \C^4$ 
consisting of points $(x_{12},x_{13},x_{23};x_{123})$ satisfying 
the single general hexagon relation (GH) for $S_0(K) \subset \C^4$: 
\[
x_{123}^2=
\frac{1}{2}
\left|
\begin{array}{ccc}
x_{11} & x_{12} & x_{13}\\
x_{21} & x_{22} & x_{23}\\
x_{31} & x_{32} & x_{33}
\end{array}
\right|.  
\]
It follows that every point of $F_2(K) \subset \C^3$ lifts to $S_0(K) \subset \C^4$ 
and hence $K$ admits no ghost characters, where $K$ is a knot in 3-bridge position. 

The same argument applies to $2$-bridge knots as well.
Therefore, any knot with bridge index less than $4$ admits no ghost characters.
\end{proof}

Theorem \ref{thm_noghost} and Proposition \ref{prop_conj}, together with 
Theorem \ref{thm_Ng_conj} (1) proved below, imply that Conjecture \ref{conj_nag} holds 
for all $2$-bridge\footnote{The $2$-bridge knot case was originally established in \cite{Ng2}.
See also \cite{Nagasato, Nagasato3, Nagasato0} for alternative proofs.} 
and $3$-bridge knots. 
We also remark that the method used in the proof of Theorem \ref{thm_noghost}
to compute $F_2(K)$ and $S_0(K)$ from a bridge presentation 
was originally developed for braid presentations in \cite{Nagasato3}. 


\subsection{Ghost characters and obstructions to Conjecture \ref{conj_nag}}\label{subsec_ng}
We conclude this paper by giving a criterion, formulated in terms of ghost characters, 
for Conjecture \ref{conj_nag} to hold and for the map $\widehat{\Phi}$ to be surjective.

\begin{theorem}\label{thm_Ng_conj}
Let $K$ be a knot with an $n$-crossing diagram. Then the following hold. 
\begin{enumerate}
\item If $K$ admits no ghost characters, then the map 
\[
\widehat{\Phi} : S_0(K) \longrightarrow X(\Sigma_2K) 
\]
is surjective. Consequently, Conjecture $\ref{conj_nag}$ holds for $K$; 
in particular, this is the case for all $2$-bridge
and $3$-bridge knots by Theorem $\ref{thm_noghost}$.

\item If $K$ admits a ghost character $\mathbf{g} \in F_2(K)$ such that 
\[
(h^*)^{-1}(\mathbf{g}) \neq \emptyset, 
\]
then the map $\widehat{\Phi}$ is not surjective. 

\item If $K$ admits a ghost character $\mathbf{g} \in F_2(K)$ such that 
\[
(h^*)^{-1}(\mathbf{g}) = \emptyset,
\]
then Conjecture $\ref{conj_nag}$ fails. 
\end{enumerate}
\end{theorem}

\begin{proof}
Suppose that $X(\Sigma_2K)$ is parametrized by $(z_{ab};y_{def}) \in \C^{{n \choose 2}+{n-1 \choose 3}}$ 
as in Subsection \ref{subsec_ghost}. 
Regarding (1), we focus on the following quadratic relations for $y_{def}$ $(2 \leq a<b<c \leq n)$, 
which any point in $X(\Sigma_2K) \subset \C^{{n \choose 2}+{n-1 \choose 3}}$ 
must satisfy (see Subsection \ref{subsec_gam}): 
\begin{eqnarray*}
\mathbf{(P1)} &(y_{abc})^2-P_{abc} y_{abc}+Q_{abc}=0,&\\
\mathbf{(P4)} 
&(y_{234}-y_{243})(2y_{abc}+z_{1a} z_{1b} z_{1c}-z_{1a} y_{bc}-z_{1b} y_{ac}-z_{1c} y_{ab})
-   
\left|\begin{array}{cccc}
z_{12} & y_{2a} & y_{2b} & y_{2c}\\
z_{13} & y_{3a} & y_{3b} & y_{3c}\\
z_{14} & y_{4a} & y_{4b} & y_{4c}\\
2      & z_{1a}  & z_{1b} & z_{1c}
\end{array}\right|=0,&
\end{eqnarray*}
where $y_{ij}=z_{1i}z_{1j}-z_{ij}$ and  
\begin{eqnarray*}
P_{abc}&=&z_{1a}y_{bc}+z_{1b}y_{ac}+z_{1c}y_{ab}-z_{1a}z_{1b}z_{1c},\\
Q_{abc}&=&z_{1a}^2+z_{1b}^2+z_{1c}^2+y_{ab}^2+y_{ac}^2+y_{bc}^2+y_{ab}y_{ac}y_{bc}
-z_{1a}z_{1b}y_{ab}-z_{1a}z_{1c}y_{ac}-z_{1b}z_{1c}y_{bc}-4,\\
y_{243}&=&-y_{234}-z_{12}z_{13}z_{14}+z_{12}y_{34}+z_{13}y_{24}+z_{14}y_{23}.
\end{eqnarray*}
Note that in the present setting, some coefficients in (P1) and (P4) 
differ in sign from those in Subsection \ref{subsec_gam}, 
since $X(\Sigma_2 K)$ is parametrized by positive traces 
\[
z_{ab}(\chi_{\rho_*})=t_{m_am_b}(\chi_{\rho_*}),\ y_{abc}(\chi_{\rho_*})=t_{m_am_bm_c}(\chi_{\rho_*}).
\] 
Fix a point $(x_{ab}) \in F_2(K)$ for the coefficients in (P1) and (P4). 
If every relation in (P1) has a double root $y_{abc}$ $(2 \le a < b < c \le n)$, 
then (P1) and (P4) admit at most one common solution $(x_{ab}; y_{abc}) \in X(\Sigma_2 K)$. 
In the remaining cases, relabeling the meridians of $K$ if necessary, 
we may assume without loss of generality that the equation 
\[
(y_{234})^2 - P_{234} y_{234} + Q_{234} = 0 
\]
does not have a double root. 
Then, by Fricke's lemma (see \cite{Mag}, for example), 
if $y_{234}$ is a root of the above equation, then $y_{243}$ is the other root. 
Hence $y_{234} - y_{243} \neq 0$, and (P1) and (P4) admit at most two common solutions 
$(x_{ab}; y_{def})$, with $y_{def}$ chosen consistently. 
These show that $(h^*)^{-1}(x_{ab})$ consists of at most two points. 
In particular, if $(x_{ab})$ is not a ghost character, 
then $(h^*)^{-1}(x_{ab})$ consists of exactly one point. 
Indeed, a point ${\mathbf x}=(x_{ab}) \in {\rm Im}(q) \subset F_2(K)$ lifts via the map $r$ to
\[
\widehat{\Phi}(\chi_{\rho})
=\left(x_{ab}; x_{1d}x_{1e}x_{1f}-\frac{1}{2}(x_{1d}x_{ef}+x_{1e}x_{df}+x_{1f}x_{de})\right) \in X(\Sigma_2K), 
\]
where $\chi_{\rho}=(x_{ab};x_{ijk})$ is a trace-free character in the preimage 
$q^{-1}(\mathbf{x}) \subset S_0(K)$. For this character, the value 
\[
y_{def}=x_{1d}x_{1e}x_{1f}-\frac{1}{2}(x_{1d}x_{ef}+x_{1e}x_{df}+x_{1f}x_{de}). 
\]
is the double root of each quadratic equation for $y_{def}$ in (P1). 
This follows from the fact that the discriminant $D:= (P_{def})^2 - 4Q_{def}$ 
vanishes at $\widehat{\Phi}(\chi_{\rho})$, as shown below: 
\begin{eqnarray*}
&&P_{def}=2\left(x_{1d}x_{1e}x_{1f}-\frac{1}{2}(x_{1d}x_{ef}+x_{1e}x_{df}+x_{1f}x_{de})\right),\\
&&Q_{def}=\left(x_{1d}x_{1e}x_{1f}-\frac{1}{2}(x_{1d}x_{ef}+x_{1e}x_{df}+x_{1f}x_{de})\right)^2
-\frac{1}{4}D^{1def}_{1def}({\mathbf x}). 
\end{eqnarray*}
Here $D^{1def}_{1def}({\mathbf x})=0$, 
since $D^{1def}_{1def}({\mathbf x})$ is a sister relation of (R) for $S_0(K)$. 
Therefore, for every point in $\mathrm{Im}(q)$, its preimage under $h^*$ 
(i.e., the fiber of $h^*$ over it) consists of a single point that lies in 
$\mathrm{Im}(\widehat{\Phi})$.  
(A similar result are shown in \cite[Proposition 6]{Nagasato-Yamaguchi}.)  
As a consequence, any point of $X(\Sigma_2K)$ outside $\mathrm{Im}(\widehat{\Phi})$ 
must be mapped by $h^*$ to a point outside $\mathrm{Im}(q)$, namely, a ghost character of $K$. 
This completes the proof of (1). 

Regarding (2), suppose that $K$ admits a ghost character 
$\mathbf{g}=(x_{ij}) \in F_2(K)$ such that $(h^*)^{-1}(\mathbf{g}) \neq \emptyset$. 
Then there exists a representation
\[
\rho_* : \pi_1(\Sigma_2K) \to \SL_2(\C)
\]
satisfying $t_{m_i m_j}(\rho_*) = x_{ij}$ for all $i<j$. 
In this case, the character $\chi_{\rho_*}$ does not lie in $\mathrm{Im}(\widehat{\Phi})$. 
Otherwise, there exists a trace-free character $\chi_{\rho}=(x_{ij};x_{ijk}) \in S_0(K)$ such that 
$\widehat{\Phi}(\chi_{\rho})=\chi_{\rho_*}$. 
Then the image $q(\chi_{\rho})=(x_{ij})$ coincides with $\mathbf{g}$, 
contradicting the fact that $\mathbf{g}$ is a ghost character. 
Consequently, the map $\widehat{\Phi}$ is not surjective.

Regarding (3), suppose that $K$ admits a ghost character $\mathbf{g} \in F_2(K)$ 
such that $(h^*)^{-1}(\mathbf{g}) = \emptyset$. Then the map $h^*$ fails to be surjective, 
and therefore cannot be an isomorphism. 
Consequently, Conjecture \ref{conj_nag} does not hold. 
\end{proof}

Recall that, in the proof of Theorem \ref{thm_Ng_conj} (1), 
for a fixed point $(x_{ab}) \in F_2(K)$, the system of equations (P1) and (P4) admits at most two 
common solutions. This observation leads to the following characterization of Ng's conjecture.

\begin{theorem}[A necessary and sufficient condition for Ng's conjecture]\label{iff}
Ng's conjecture holds for a knot $K$ if and only if $K$ admits no ghost characters.
\end{theorem}

\begin{proof}
The sufficient condition for Ng's conjecture has already been shown in Theorem \ref{thm_Ng_conj} (1). 
To prove the necessary condition, we analyze the discriminant $D$ of (P1):
\[
(y_{abc})^2 - P_{abc}y_{abc} + Q_{abc} = 0.
\]
We begin by rewriting $P_{abc}$ and $Q_{abc}$ in terms of the parameters $z_{ij}$:
\begin{eqnarray*}
P_{abc}&=&2z_{1a} z_{1b} z_{1c}-z_{1a} z_{bc}-z_{1b} z_{ac}-z_{1c} z_{ab}\\
Q_{abc}&=&z_{1a}^2 z_{1b}^2 z_{1c}^2-z_{1a}^2 z_{1b} z_{1c} z_{bc}
-z_{1a} z_{1b}^2 z_{1c} z_{ac}-z_{1a} z_{1b} z_{1c}^2 z_{ab} + z_{1a} z_{1b} z_{ac} z_{bc}\\
&&+z_{1a} z_{1c} z_{ab} z_{bc}+z_{1b} z_{1c} z_{ab} z_{ac} - z_{1a} z_{1b} z_{ab}-z_{1a}z_{1c} z_{ac}
-z_{1b} z_{1c} z_{bc}-z_{ab} z_{ac} z_{bc}\\
&&+ z_{1a}^2 + z_{1b}^2 + z_{1c}^2 + z_{ab}^2 + z_{ac}^2 + z_{bc}^2 - 4
\end{eqnarray*}
In this setting, one can verify that the discriminant $D$ of the quadratic equation (P1) 
for a fixed point ${\mathbf x}=(x_{ij}) \in F_2(K)$ satisfies 
\[
D=P_{abc}^2-4Q_{abc}=D_{1abc}^{1abc}({\mathbf x}). 
\]
Consequently, the quadratic equation (P1) has a double root 
for the fixed point $\mathbf{x} \in F_2(K)$ if and only if the relation 
$D_{1abc}^{1abc}({\mathbf x})=0$ holds.

Let $K$ be a knot that admits a ghost character $\mathbf{g}=(x_{ij}) \in F_2(K)$. 
We claim that the preimage 
$(h^*)^{-1}(\mathbf{g})$ consists of either exactly two points or is empty. 
Assume that $(h^*)^{-1}(\mathbf{g}) \neq \emptyset$, so that there exists 
a representation $\rho_* : \pi_1(\Sigma_2K) \to \SL_2(\C)$ whose character 
lies in $(h^*)^{-1}(\mathbf{g})$. This means that $(z_{ij}(\chi_{\rho_*}))=(x_{ij})=\mathbf{g}$. 
From such a representation $\rho_*$, we construct another representation 
$\bar{\rho_*} : \pi_1(\Sigma_2K) \to \SL_2(\C)$ by setting 
\[
\bar{\rho_*}(m_1m_i):={}^T\hspace*{-0.1cm}\rho_*(m_1m_i)\ 
(2 \leq i \leq n), 
\]
where ${}^T\hspace*{-0.1cm}*$ denotes the transpose of a matrix $*$,
and extending this assignment multiplicatively to all elements of $\pi_1(\Sigma_2K)$. 
This definition yields a well-defined group homomorphism. 
Indeed, we can check that the relators of $\pi_1(\Sigma_2K)$ are preserved by $\bar{\rho_*}$, 
using the presentation 
\[
\pi_1(\Sigma_2K) \cong \langle m_1m_i\ (2 \leq i \leq n) \mid w(r_j),w(m_1r_jm_1^{-1})\ 
(1 \leq j \leq n-1), m_i^2\ (1 \leq i \leq n)\rangle 
\] 
associated with the Wirtinger presentation of the knot group 
\[
G(K)=\langle m_1,\cdots,m_n \mid r_1, \cdots, r_{n-1} \rangle. 
\] 
For instance, a relator of the form $r_j = m_a m_b m_a^{-1} m_c^{-1}$ of $G(K)$ 
gives the corresponding relators in the presentation of $\pi_1(\Sigma_2K)$:  
\begin{eqnarray*}
w(r_j)&=&(m_1m_a)^{-1}(m_1m_b)(m_1m_a)^{-1}(m_1m_c),\\ 
w(m_1r_jm_1^{-1})&=&(m_1m_a)(m_1m_b)^{-1}(m_1m_a)(m_1m_c)^{-1}. 
\end{eqnarray*}
We compute the images of these words under $\bar{\rho_*}$ as follows: 
\begin{eqnarray*}
\bar{\rho_*}(w(r_j))&=&\bar{\rho_*}((m_1m_a)^{-1}(m_1m_b)(m_1m_a)^{-1}(m_1m_c))\\
&=&\bar{\rho_*}((m_1m_a)^{-1})\bar{\rho_*}(m_1m_b)\bar{\rho_*}((m_1m_a)^{-1})
\bar{\rho_*}(m_1m_c)\\
&=&{}^T\hspace*{-0.1cm}\rho_*(w(m_1r_jm_1^{-1})^{-1})=E,\\
\bar{\rho_*}(w(m_1r_jm_1^{-1}))&=&\bar{\rho_*}((m_1m_a)(m_1m_b)^{-1}(m_1m_a)(m_1m_c)^{-1})\\
&=&\bar{\rho_*}((m_1m_a))\bar{\rho_*}((m_1m_b)^{-1})\bar{\rho_*}((m_1m_a))
\bar{\rho_*}((m_1m_c)^{-1})\\
&=&{}^T\hspace*{-0.1cm}\rho_*(w(r_j)^{-1})=E. 
\end{eqnarray*}
The remaining cases are shown in the same way. 

In this setting, we show that the characters $\chi_{\rho_*}$ and $\chi_{\bar{\rho_*}}$ 
define distinct points of $X(\Sigma_2K)$. Again, by Fricke's Lemma, 
if $y_{abc}(\chi_{\rho_*})$ is a solution of the quadratic equation (P1):  
\[
(y_{abc})^2-P_{abc}y_{abc}+Q_{abc}=0, 
\]
then $y_{acb}(\chi_{\rho_*})$ is the other solution.  
Accordingly, if $\rho_*$ is a representation such that $y_{abc}(\chi_{\rho_*})$ is a solution of (P1)  
for each $a<b<c$, then the associated representation $\bar{\rho_*}$ yields the other solution, 
namely $y_{abc}(\chi_{\bar{\rho_*}})=y_{acb}(\chi_{\rho_*})$. 
This correspondence can be verified directly: 
\begin{eqnarray*}
z_{1a}(\chi_{\bar{\rho_*}})&=&\tr(\bar{\rho_*}(m_1m_a))=z_{1a}(\chi_{\rho_*}),\\
z_{ab}(\chi_{\bar{\rho_*}})&=&\tr(\bar{\rho_*}((m_1m_a)^{-1}(m_1m_b)))=z_{ab}(\chi_{\rho_*}),\\
y_{abc}(\chi_{\bar{\rho_*}})&=&\tr(\bar{\rho_*}((m_1m_a)(m_1m_b)(m_1m_c)))=y_{acb}(\chi_{\rho_*}). 
\end{eqnarray*}

Since $\mathbf{g}$ is a ghost character with nonempty preimage under $h^*$, 
$\mathbf{g}$ does not satisfy at least one of the relations $D_{1abc}^{1abc}=0$. 
Indeed, suppose to the contrary that $D_{1abc}^{1abc}(\mathbf{g})=0$ holds for all triples $(a,b,c)$.
Then Lemma 2.1 in \cite{Mag} implies the existence of a trace-free matrix 
$A \in \SL_2(\C)$ such that 
\[
(A \cdot \rho_*(m_1m_i))^2=-E\ (2 \leq i \leq n), 
\] 
unless $\tr([\rho_*(m_1m_i),\rho_*(m_1m_j)])=2$ for all $2 \leq i < j \leq n$.
Although Lemma 2.1 is stated for a fixed triple $(a,b,c)$,
a basic linear-algebraic argument shows that the same matrix $A$ 
can be chosen uniformly for all triples. 

Suppose that $\tr([\rho_*(m_1m_i),\rho_*(m_1m_j)])=2$ for all $2 \leq i < j \leq n$. 
Then, by \cite[Theorem 1.2]{Flo}, together with the assumption that $D_{1abc}^{1abc}(\mathbf{g})=0$ 
for all triples $(a,b,c)$, it follows that $\rho_*$ is reducible.
Then there exsits the character $\chi_{\rho_{\rm ab}}$ of an abelian representation 
$\rho_{\rm ab} : \pi_1(\Sigma_2K) \to \SL_2(\C)$ such that $\chi_{\rho_*}=\chi_{\rho_{\rm ab}}$; 
these characters correspond to a bifurcation point in $X(\Sigma_2K)$. 
By Proposition 6 and Lemma 22 in \cite{Nagasato-Yamaguchi}, 
the character $\chi_{\rho_{\rm ab}}$ lies in the image of $\widehat{\Phi}$. 
However, this contradicts the conclusion shown in the proof of Theorem \ref{thm_Ng_conj} (2), 
namely that the preimage $(h^*)^{-1}(\mathbf{g})$ lies outside $\mathrm{Im}(\widehat{\Phi})$.
We therefore exclude this case and assume that $\tr([\rho_*(m_1m_i),\rho_*(m_1m_j)]) \neq 2$ 
for at least one pair $(i,j)$ with $2 \leq i < j \leq n$. In particular, $\rho_*$ is irreducible. 

In this situation, the trace-free matrix $A \in \SL_2(\C)$ introduced above indeed exists.
Using $A$ and $\rho_*$, we define a trace-free representation 
$\rho_A : G(K) \to \SL_2(\C)$ by
\[
\rho_A(m_i) := A \cdot \rho_*(m_1m_i) \ (1 \leq i \leq n).
\]
A direct calculation using 
\[
A \cdot \rho_*(m_1m_i)=- \rho_*(m_1m_i)^{-1} \cdot A^{-1}\ (2 \leq i \leq n)
\]
shows that $\rho_A$ is well-defined, that is, $\rho_A$ preserves the relators of $G(K)$. 
In this setting, the projection $q(\chi_{\rho_A})$ of the trace-free character 
$\chi_{\rho_A} \in S_0(K)$ coincides with the ghost character $\mathbf{g}=(x_{ij})$, 
a contradiction. 

Consequently, there exists at least one triple $(a,b,c)$ such that 
$D_{1abc}^{1abc}(\mathbf{g}) \neq 0$. 
Since the discriminant $D=D_{1abc}^{1abc}(\mathbf{g})$ is nonzero, we have 
$y_{abc}(\chi_{\rho_*}) \neq y_{abc}(\chi_{\bar{\rho_*}})$.  
Hence the characters $\chi_{\rho_*}$ and $\chi_{\bar{\rho_*}}$ 
define distinct points of $(h^*)^{-1}(\mathbf{g}) \subset X(\Sigma_2K)$. 
Accordingly, for any ghost character $\mathbf{g}$, 
the preimage $(h^*)^{-1}(\mathbf{g})$ consists of either exactly two points or is empty. 
In particular, the map $h^*$ fails to be bijective. 
Therefore, if $K$ admits a ghost character, then Ng's conjecture does not hold for $K$.
\end{proof}

From this viewpoint, we have focused on ghost characters of knots and 
their applications. 
Computational evidence suggests that certain $4$-bridge and $5$-bridge knots 
admit ghost characters. By Theorem \ref{iff}, such knots provide counterexamples 
to Conjecture \ref{conj_nag}. A detailed investigation of these examples will appear 
in \cite{Nagasato-Suzuki}. 


\section*{Acknowledgments}
I would like to thank Yoshikazu Yamaguchi for many helpful comments on ghost
characters of knots. 
Parts of the results in this paper were presented in my talk at the conference
``RIMS Seminar: Representation spaces, twisted topological invariants, and 
geometric structures of $3$-manifolds,'' held in 2012 (see \cite{Nagasato0}). 
I am grateful to the organizers for their warm hospitality.
This research was started during my stay at the University of California, 
Riverside (2004-2006). 
I am deeply indebted to Professor Xiao-Song Lin for his invaluable discussions, 
comments, encouragement, and generous hospitality throughout that period. 
The research carried out at that time was published in \cite{Nagasato3} and 
forms the basis of the present work. 
The early stages of this research were partially supported by JSPS Research
Fellowships for Young Scientists, JSPS KAKENHI for Young Scientists (Start-up),
MEXT KAKENHI for Young Scientists (B), and JSPS KAKENHI for Young Scientists (B).
The current research was partially supported by JSPS KAKENHI (C),
Grant No. 20K03619.



\begin{thebibliography}{99}
\bibitem{Bullock} D. Bullock: 
{\em Rings of $SL_2(\C)$-characters and the Kauffman bracket skein module}, 
Comment. Math. Helv. {\bf 72} (1997), 521--542.

\bibitem{Bullock2} D. Bullock: 
{\em A finite set of generators for the Kauffman bracket skein algebra}, 
Math. Z. {\bf 231} (1999), 91--101. 

\bibitem{Burde} G. Burde: 
{\em Darstellungen von Knottengruppen},
Math. Ann. {\bf 173} (1967), 24--33. 

\bibitem{Burde-Zieschang} G. Burde and H. Zieschang: 
{\em Knots}, de Gruyter Studies in Mathematics {\bf 5}, Walter de Gruyter \& Co., 
Berlin, 2003

\bibitem{Cornwell} C.R. Cornwell: 
{\em Character varieties of knot complements and branched double-covers via the cord ring},
preprint, arXiv:1509.04962.

\bibitem{Culler-Shalen} M. Culler and P. Shalen: 
{\em Varieties of group presentations and splittings of $3$-manifolds}, 
Ann. of Math. {\bf 117} (1983), 109--146.

\bibitem{deRham} G. de Rham: 
{\em Introduction aux polyn\^{o}mes d'un n{\oe}ud}, 
Enseign. Math. {\bf 13} (1967), 187--194. 

\bibitem{EENS} T. Ekholm, J. Etnyre, L. Ng and M. Sullivan: 
{\em Knot contact homology}, arXiv:1109.1542v2. 

\bibitem{Flo} C. Florentino:
{\em Invariants of $2 \times 2$ matrices, irreducible $\SL(2,\C)$-characters 
and the Magnus trace map},
Geom. Dedicata {\bf 121}, 167--186 (2006).

\bibitem{Fox} R. Fox: 
{\em Free differential calculus III, subgroups}, 
Ann. of Math. {\bf 64} (1956).

\bibitem{Fricke} R. Fricke and Klein: 
{\em Vorlesungen \:{u}ber die Theorie der automorphen Functionen},
Vol. 1, B.G. Teubner, Leipzig, 1897.

\bibitem{Gelca-Nagasato} R. Gelca and F. Nagasato: 
{\em Some results about the Kauffman bracket skein module of 
the twist knot exterior}, 
J. Knot Theory Ramifications {\bf 15} (2006), 1095--1106. 

\bibitem{Gonzalez-Montesinos} F. Gonz\'{a}lez-Acu\~{n}a and J.M. Montesinos: 
{\em On the character variety of group representations in $\SL(2,\C)$ 
and $\mathrm{PSL}(2,\C)$}, Math. Z., {\bf 214} (1993), 627--652.

\bibitem{Horowitz} R. Horowitz: 
{\em Characters of free groups represented in the two dimensional linear group}, 
Comm. Pure Appl. Math. {\bf 25} (1972), 635--649.

\bibitem{Kawauchi} A. Kawauchi: 
{\em A survey of knot theory}, Birkh\"{a}user Verlag, Basel, 1996. 

\bibitem{Kinoshita} S. Kinoshita: 
{\em Isoukikagaku-nyumon (in Japanese)}, 
Baifukan, Tokyo, 2000.

\bibitem{Lin1}X.-S. Lin:
{\em A knot invariant via representation spaces}, 
Jour. of Diff. Geom. {\bf 35} (1992), 337--357.

\bibitem{Lin2}
X.-S. Lin: 
{\em Representations of knot groups and twisted Alexander polynomials}, 
Acta Math. Sin. (Engl. Ser.) {\bf 17} (2001), 361--380.

\bibitem{Mag}
W. Magnus: 
{\em Rings of Fricke characters and automorphism groups of free groups}, 
Math. Z. {\bf 170} (1980), 91--103. 

\bibitem{Nagasato} F. Nagasato: 
{\em Character variety no danmen kara yudou sareru daisutayoutai-zoku to 
knot contact homology 
(in Japanese)}, Proceedings of the 54th Topology Symposium, Aizu, Japan (2007).

\bibitem{Nagasato3} F. Nagasato:
{\em Algebraic varieties via a filtration of the KBSM and knot contact homology}, 
Topology Appl., {\bf 264} (2019), 251-275.  

\bibitem{Nagasato2} F. Nagasato: 
{\em Finiteness of a section of the $\SL(2,\C)$-character variety of knot groups},
Kobe J. Math., {\bf 24} (2007), 125--136. 

\bibitem{Nagasato1} F. Nagasato: 
{\em On a behavior of a slice of the $\SL_2(\C)$-character variety of a knot group 
under the connected sum}, Topology Appl. {\bf 157} (2010), 182--187 (2010). 

\bibitem{Nagasato4} F. Nagasato: 
{\em On the trace-free characters}, 
RIMS Kokyuroku ``Representation spaces, twisted topological invariants 
and geometric structures of 3-manifolds'' {\bf 1836} (2013), 110--123.

\bibitem{Nagasato0} F. Nagasato: 
{\em On trace-free characters and abelian knot contact homology}, 
RIMS Seminar ``Representation spaces, twisted topological invariants and 
geometric structures of 3-manifolds'' (28th May-1st June 2012, Gora Seiunso, Hakone, Japan. 

\bibitem{Nagasato-Suzuki} F. Nagasato and S. Suzuki:
{\em Trace-free characters and abelian knot contact homology II}, preprint. 

\bibitem{Nagasato-Yamaguchi} F. Nagasato and Y. Yamaguchi: 
{\em On the geometry of the slice of trace-free $\SL_2(\C)$-characters of a knot group}, 
Math. Ann. {\bf 354} (2012), 967--1002.

\bibitem{Ng1} L. Ng:
{\em Knot and braid invariants from contact homology I},
Geom. Topol. {\bf 9} (2005), 247--297. 

\bibitem{Ng2} L. Ng:
{\em Knot and braid invariants from contact homology II},
Geom. Topol. {\bf 9} (2005), 1603-1637. 

\bibitem{Przytycki1} J.H. Przytycki: 
{\em Fundamentals of Kauffman bracket skein module}, 
Kobe J. Math. {\bf 16} (1999), 45--66. 

\bibitem{Przytycki2} J.H. Przytycki: 
{\em Skein modules of 3-manifolds}, 
Bull. Pol. Acad. Sci. {\bf 39} (1991), 91--100.

\bibitem{Przytycki-Sikora} J.H. Przytycki and A. Sikora: 
{\em Skein algebra of a group}, 
Banach Center Publications {\bf 42} (1998), 297--306.

\bibitem{Przytycki-Sikora2} J.H. Przytycki and A. Sikora: 
{\em On skein algebras and $\SL_2(\C)$-character varieties}, 
Topology {\bf 39} (2000), 115--148.

\bibitem{Vogt} H. Vogt: 
{\em Sur les invariants fondamentaux des \'{e}quations differentielles lin\'{e}aires du second 
ordre}, Ann. Sci. \'{E}col. Norm. Sup\'{e}r, III. Ser. {\bf 6} (1889), 3--72. 

\bibitem{z} R. Zentner, {\em Representation spaces of pretzel knots}, 
Algebr. Geom. Topol. {\bf 11} (2011), 2941-2970.
\end{thebibliography}
\end{document}